\documentclass[12pt]{amsart}
\usepackage[margin=1in]{geometry}
\usepackage{amsmath,amsfonts,amsthm,amssymb,bbm}
\usepackage{graphicx,color,dsfont}
\usepackage{fourier}

\newtheorem{theorem}{Theorem}
\newtheorem{Problem}{Problem}
\newtheorem{lemma}{Lemma}
\newtheorem{proposition}{Proposition}

\newtheorem{corollary}{Corollary}

\newtheorem{claim}{Claim}

 \newtheorem{thm}{Theorem}[section]
 \newtheorem{cor}[thm]{Corollary}
 \newtheorem{lem}[thm]{Lemma}
 \newtheorem{prop}[thm]{Proposition}
 \theoremstyle{definition}
 
 \theoremstyle{remark}

 \numberwithin{equation}{section}

\newcommand{\vertiii}[1]{{\left\vert\kern-0.25ex\left\vert\kern-0.25ex\left\vert #1
    \right\vert\kern-0.25ex\right\vert\kern-0.25ex\right\vert}}

\newcommand{\f}[2]{\frac{#1}{#2}}

\newcommand{\cl}{{\mathcal L}}

\newcommand{\al}{\alpha}

\newcommand{\ga}{\gamma}

\newcommand{\de}{\delta}

\newcommand{\De}{\Delta}

\newcommand{\ka}{\kappa}
\newcommand{\la}{\lambda}

\newcommand{\La}{\Lambda}
\newcommand{\si}{\sigma}

\newcommand{\rn}{{\mathbf R}^n}

\newcommand{\rone}{\mathbf R}
\newcommand{\rtwo}{\mathbf R^2}

\newcommand{\ku}{\mathbf u}

\newcommand{\dpr}[2]{\langle #1,#2 \rangle}

\newcommand{\eps}{\epsilon}

\newcommand{\cq}{\mathcal Q}
\newcommand{\cp}{\mathcal P}

\newcommand{\p}{\partial}

\newcommand{\beq}{\begin{equation}}
\newcommand{\eeq}{\end{equation}}
\newcommand{\beqna}{\begin{eqnarray*}}
\newcommand{\eeqna}{\end{eqnarray*}}
\newcommand{\beqn}{\begin{equation*}}
\newcommand{\eeqn}{\end{equation*}}
\newcommand{\bp}{\begin{proof}}
\newcommand{\ep}{\end{proof}}
\newcommand{\bprop}{\begin{proposition}}
\newcommand{\eprop}{\end{proposition}}
\newcommand{\bt}{\begin{theorem}}
\newcommand{\et}{\end{theorem}}
\newcommand{\bex}{\begin{Example}}
\newcommand{\eex}{\end{Example}}
\newcommand{\bc}{\begin{corollary}}
\newcommand{\ec}{\end{corollary}}
\newcommand{\bcl}{\begin{claim}}
\newcommand{\ecl}{\end{claim}}
\newcommand{\bl}{\begin{lemma}}
\newcommand{\el}{\end{lemma}}

\newcommand{\ck}{{\mathcal K}}

\begin{document}

\title[Steady states  of the forced SQG are globally attractive]
 {On the forced surface quasi-geostrophic  equation: existence of steady states and sharp relaxation rates}

\thanks{Stefanov  is partially  supported by  NSF-DMS   under  \# 1908626.}

 \author[F. Hadadifrad]{\sc Fazel Hadadifard}
 \address{ Department of Mathematics,
 	Drexel University,
 Korman Center, 
 33rd \& Market Streets, Philadelphia, PA 19104,  USA}
 \email{fh352@drexel.edu}
 
\author[A. Stefanov]{\sc Atanas G. Stefanov}
\address{Department of Mathematics,
University of Kansas,
1460 Jayhawk Boulevard,  Lawrence KS 66045--7523, USA}
\email{stefanov@ku.edu}



\date{\today}

\begin{abstract}
 We consider the asymptotic behavior of the surface quasi-geostrophic equation, subject to a small external force. Under suitable assumptions on the forcing, we first construct the steady states and we provide a number of useful {\it a posteriori} estimates for them. Importantly, to do so, we only impose minimal cancellation conditions on the forcing function. 
 
 Our main result is that all  $L^1\cap L^\infty$   localized initial data produces global  solutions of the forced SQG, which converge to the steady states   in $L^p(\rtwo), 1<p\leq 2$  as time goes to infinity.   This establishes that the steady states serve as one point attracting set.  Moreover,  by employing the method of scaling variables, we compute the sharp relaxation rates, by requiring slightly more localized initial data.

\end{abstract}

\maketitle

\section{Introduction}
  In this paper, the main object of investigation is   the forced two dimensional surface quasi-geostrophic equation 
  \begin{eqnarray}
  	\label{10}
  	\left\{ 
  	\begin{array}{l} 
  		\theta_t+\ku_\theta \cdot \nabla \theta+  \La^{\al} \theta= f, \  x \in \rtwo, t>0 \\ 
  		 \theta(x, 0)= \theta_0(x)
  	\end{array}  
  	\right.
  \end{eqnarray}
  where $\theta, f : \rtwo\to \rone$, $\La=\sqrt{-\De}$ is the Zygmund's  operator and  
  $$
  \ku_\theta = R^{\perp} \theta= (-R_2\theta, R_1\theta)=\La^{-1} (- \partial_2 \theta, \partial_1 \theta).
  $$
   Note that $div(\ku_\theta)=0$.  In fact, we adopt the notation $\ku_f$ for any scalar field $f$ to mean the divergence-free vector field  $\ku_f:=R^{\perp} f$. The model \eqref{10} is of fundamental importance in the modeling of large scale fluid motion, especially in oceanographic context.  The critical case, namely $\al=1$, which is also the most challenging from a mathematical standpoint,  was put forward in \cite{CMT} (see also \cite{CCW}), as a model of surface temperature of  a rapidly rotating fluid. 
   In fact, this and related models   frequently arise in fluid dynamics and as such, they  have been widely studied  in the last twenty  years. We refer the reader to the works \cite{AH, B, C, C1, CC,   NS} and references therein.

   We consider the parameter $\al$ in the sub-critical regime $\al\in (1,2)$, although the case $\al=2$ is certainly interesting as well, both from physical and mathematical point of view.  
  
  \subsection{Global regularity for \eqref{10}: some  recent results and historical perspectives} 
  The well-posedness theory for the homogeneous version of \eqref{10}, i.e. with $f=0$ is well-understood in the sub-critical case $\al>1$. Informally, reasonably localized (including  large) initial data $\theta_0$  produces global solutions, which preserve the functional-analytic properties of the initial data. That is, the so-called global regularity of the initial value problem has been established in various scenarios.  These results have appeared in literally hundreds of publications, which is why we do not attempt to follow through with precise statements and references. Similarly, in the case $\al=1$, the criticality of the problem allows one  to reproduce the global regularity problem for small data. 
  More recently,   a substantial progress has been made  in the regularity problem for large data, see \cite{CVa, CV,  KNV, KN}.  It has been established, that under fairly mild assumptions on initial data, the solution persists globally and preserve the smoothness of $\theta_0$.  It is worth noting that  the long time  dynamics  for the sub-critical and critical cases of \eqref{10}
  (both in the regime $f=0$ and $f\neq 0$) were studied intensively in \cite{CTV, CTV1, JMST, FPV, NS}. In particular, estimates for the decay rates for  regular and weak solutions were obtained in \cite{NS, SS}. In  \cite{CTV1}, the authors have established the existence of a global attractor for the problem posed on periodic domain.  
  
  The well-posedness in the supercritical case $\al<1$ remains an open elusive problem. The expectation is that at least for some initial data, one should  observe a finite time blow up. That has not been settled as of this writing.   

  \subsection{Motivation and main results}  
  Our main object of investigation is the forced problem.  
  Of particular interest will be the properties of the steady states 
   $\tilde{\theta}$, which   satisfies the following profile equation 
  \begin{eqnarray}
  \label{20}
  \La^{\al} \widetilde{\theta}+ \ku_{\widetilde{\theta}} \cdot \nabla \widetilde{\theta}= f, x\in\rtwo. 
  \end{eqnarray}
  More precisely, we would like to draw conclusions about the global dynamics of \eqref{10} from the properties of  $\tilde{\theta}$. This is indeed the main objective of this work. We should mention here that the problem that we aim at considering has already been addressed, at least partially, in several recent works. Regarding the un-forced  SQG (i.e. with $f=0$), in \cite{SS}, the authors have obtained some estimates for the decay rates of the solutions of  as well as estimates from below. More recently, in our work \cite{HS}, we have considered a wide variety of un-forced SQG like problem, of which SQG is an example. We have shown the optimal decay rates for the solutions, once the initial data $\theta_0$ has some stronger localization properties.

  We now describe the work of Dai, \cite{Dai}, which was the starting point and the main motivation of our investigation.  In it, she considers the case $1\leq \al<2$. She starts by constructing solutions    of \eqref{20}, under appropriate conditions of the small forcing term $f$. More importantly, she has established a non-linear stability property for the evolution, namely that the solution of the dynamic problem (only under the assumption that $\theta_0-\tilde{\theta}\in L^2(\rtwo)$), converges to the steady state $\tilde{\theta}$ in $L^2$ sense. Note that no estimates on the speed of the decay to zero are provided in \cite{Dai}. However, it is worth noting that even in the case of zero forcing, the convergence to zero of $ \| \theta(t, \cdot)-\tilde{\theta}\|_{L^2(\rtwo)}$ may happen  with arbitrarily slow decay, see \cite{NS}, unless one assumes more integrability of $\theta_0$.

  In order to describe our results, it is convenient to track the deviation from the steady state $\tilde{\theta}$, so we introduce $v:=\theta-\widetilde{\theta}$. This new variable   satisfies the following equation  
  \begin{eqnarray}
  \label{30}
  \left\{ 
  \begin{array}{l} 
  v_t+ \La^{\al} v+ \ku_{\widetilde{\theta}} \cdot \nabla v+  \ku_v \cdot \nabla \widetilde{\theta}+   \ku_v \cdot \nabla v= 0, \\ 
  v(x, 0)= v_0(x).
  \end{array}  
  \right.
  \end{eqnarray}
  Based on the physical interpretation of our model, we are only  interested in localized functions $\theta, \tilde{\theta}$, and consequently $f$ to work with. In addition, and for mostly the same reasons, we only consider the sub-critical case $1<\al<2$. This allows us to consider strong solutions and our results will not depend on additional assumptions on the properties of weak solutions, which is necessary in the cases $\al\leq 1$. 
 Next, we shall need to assume a sufficiently  smooth and decaying initial forcing function $f$. Note that due to the form of \eqref{20}, some cancellation of $f$ is necessary,  see Theorem \ref{theo:5}  for  the precise requirements on $f$.    
 
 We now aim at discussing   the main results  of this work. Before we present the  specifics, let us give a general overview of the goals and the general flavor of the problems  that we would like to address. Our first issue, as in \cite{Dai},  is to study the solvability of the elliptic problem \eqref{20}. This turns out to be non-trivial and we do not have a complete answer to the following natural question. 
 	\begin{Problem}
 		\label{Q1} 
 		Given smooth and decaying $f$, with appropriate cancellation conditions, construct  steady state solution 
 		$\tilde{\theta}$ of \eqref{20}. 
 		\end{Problem}
 	We note that this is in general (i.e. for large forcing $f$)  an essentially open question, which merits further, independent investigation. 	It should be stated though that in the work of Dai, \cite{Dai}, the issue was partially resolved in the case of small  forcing $f$. Even though some cancellation assumptions on the (small) forcing term $f$ are necessary, as discussed above, the conditions imposed in \cite{Dai} requires   $\hat{f}(\xi)=0: |\xi|<\delta$. This  in practice reduces the applicability of such  result,  as $f$ is forced, among other things,  to have zero moments of all orders.  We have succeeded in reducing the cancellation conditions by simply requiring that $f$ is small  in some (reasonably small) negative order Sobolev spaces,  see  Theorem \ref{theo:5} below.   
 		
Next, we are interested in the stability property of the dynamics, that is the property established in \cite{Dai} that the solutions of \eqref{10}, with any size initial data $\theta_0$ eventually converges to the steady state $\tilde{\theta}$. 
We refer to it as relaxation of the global solutions to the steady state. That is, we are asking whether or not {\it any} solution of \eqref{10} should converge/relax  to $\tilde{\theta}$,  in the appropriate norms as $t\to \infty$. 
 More precisely, 
\begin{Problem}
	\label{Q2} 
		Assuming  existence of a  solution $\tilde{\theta}$  of \eqref{20},  with appropriate properties, show that any solution of \eqref{10} converges to $\tilde{\theta}$.  Provide estimates for  the relaxation rates, possibly sharp ones. 
	\end{Problem} 
 Clearly, any result in the direction of Problem \ref{Q2} provides as a corollary, an uniqueness statement for the solvability of \eqref{20}.  Thus, a result of this type complements nicely an eventual existence result for \eqref{20}. 
 We have the following results, under the standing assumption $1<\al<2$. 
 \begin{theorem}(Existence of the steady state in unweighted spaces)
 	\label{theo:5} 
 	
 	There exists $\eps_0>$, so that whenever the forcing term 
 	$f\in \dot{W}^{-\al,\f{2}{\al-1}}(\rtwo): \|f\|_{\dot{W}^{-\al,\f{2}{\al-1}}}<\eps_0$, the steady state equation \eqref{20} has a solution $\tilde{\theta}\in L^{\f{2}{\al-1}}(\rtwo)$, with 
 	$\|\tilde{\theta}\|_{L^{\f{2}{\al-1}}(\rtwo)}\leq 2 \|f\|_{\dot{W}^{-\al,\f{2}{\al-1}}}$. 
 	If in addition, for any $p>\f{2}{3-\al}$, $f\in W^{-\al,p}(\rtwo)$, then the steady state $\tilde{\theta}\in L^p$ and it satisfies the bound 
 	$$
 	\|\tilde{\theta}\|_{L^{p}(\rtwo)}\leq 2 \|f\|_{\dot{W}^{-\al, p}}.
 	$$
 	
 Assuming $f\in \dot{W}^{-\al,\f{2}{\al-1}}(\rtwo)\cap W^{1-\al, \f{2}{\al}}:  \|f\|_{\dot{W}^{-\al,\f{2}{\al-1}}}<\eps_0$,  there is  the {\it a posteriori} estimate 
 \begin{equation}
 \label{a:60}
 \|\nabla \tilde{\theta}\|_{L^{\f{2}{\al}}}\leq C \|f\|_{\dot{W}^{1-\al, \f{2}{\al}}}.
 \end{equation}
 \end{theorem}
 \noindent 
 {\bf Remarks:} 
 \begin{itemize}
 	\item The smallness assumptions are in scale invariant spaces, as is customary. 
 	\item It is possible to formulate  an  uniqueness statement for the small solutions $\tilde{\theta}$ obtained in Theorem \ref{theo:5}, but we have stronger dynamics statement, see Theorem \ref{theo:10} below, which imply global uniqueness.
 	\item By far, the strongest cancellation condition is the requirement $f\in \dot{W}^{-\al,\f{2}{\al-1}}(\rtwo)$, which is a fairly mild one  for values of $\al$ close to $1$. In fact, this  may be Sobolev controlled by regular $L^q(\rtwo), 1<q$ norms. 
 	Even when $\al$ is close to two, our assumptions will be satisfied, at small frequencies,  by requiring the mild cancellation condition $|\hat{f}(\xi)|\leq C |\xi|^{1+\de}, |\xi|<1$.  
 	 \end{itemize}

 \begin{theorem}(Relaxation in $L^p$  spaces)
 	\label{theo:10} 
 	
 Let $1< \al < 2$ and $f\in W^{1-\al, \f{4}{2+\al}}$.  Then, there exists $\eps_0>0$, so that whenever 	  the steady state 
 $\tilde{\theta}$  satisfies $\|\nabla \tilde{\theta}\|_{L^{\f{2}{\al}}}<\eps_0$, 
and  the initial data $v_0=\theta_0-\tilde{\theta}  \in L^1\cap L^\infty(\rtwo)$, the problem \eqref{30} has an unique, global solution in  $L^2 \cap L^\infty$.  Moreover,     
 	there is a constant $C=C_{\al,  \eps_0, \|v_0\|_{L^2\cap L^\infty}}$ so that ,  
 	\begin{equation}
 	\label{14}
 	\|\theta(t, \cdot)-\tilde{\theta}(\cdot)\|_{L^p(\rtwo)} \leq  
 	\f{C} {(1+t)^{\f{2}{\al}(1-\f{1}{p})}}, \ \ 1<p\leq 2.
 	\end{equation} 
 	
 	The bound \eqref{14} can  be extended to any $2<p<\infty$, as follows. For any $q:2<q<\infty$, there exists $\eps_0(q)$, so that whenever 
 	$\|\nabla \tilde{\theta}\|_{L^{\f{2}{\al}}}<\eps_0(q)$, and $v_0\in L^1\cap L^\infty(\rtwo)$, then 
 	$$
 		\|\theta(t, \cdot)-\tilde{\theta}(\cdot)\|_{L^p(\rtwo)} \leq  
 		\f{C} {(1+t)^{\f{2}{\al}(1-\f{1}{p})}}, \ \ 2\leq p<q.
 	$$
 \end{theorem}
 \noindent 
 {\bf Remarks:} 
 \begin{itemize}
 	\item The smallness condition $\|\nabla \tilde{\theta}\|_{L^{\f{2}{\al}}}<<1$ is guaranteed by Theorem \ref{theo:5} so long as we assume 
 	$\|f\|_{\dot{W}^{1-\al, \f{2}{\al}}}<<1$. 
 	\item It is important to emphasize that $v_0$ is not assumed to be small. That is, Theorem \ref{theo:10} is a true relaxation statement.  That is,  $\tilde{\theta}$ serves as one point attractor for the evolution of \eqref{10}. 
 	\item There are  much more precise results, particularly if one assumes $v_0\in L^p\cap L^2$ instead of  $v_0\in L^1\cap L^\infty$. In this sense, Theorem \ref{theo:10} is a representative corollary of these estimates. The interested reader is invited to consult Section \ref{sec:4}. 
 	\item Related to the previous point, we have a result for initial data $v_0\in L^2$ (but not in any other $L^p$ space), which does not guarantee any decay. See Corollary \ref{cor:1} or more precisely \eqref{90}. This is in line with the results in \cite{NS}, which establish that there might be arbitrarily slow decaying to zero solutions, when $f=0$.  
 		\item The estimate \eqref{14} provides a stronger uniqueness result for the stationary problem \eqref{20} as discussed earlier. Indeed, assuming that there is another solution of \eqref{20}, $\tilde{\theta}_1\in L^1\cap L^\infty$, not necessarily small, then it needs to satisfy \eqref{14}, which implies uniqueness. 
 	
 \end{itemize}
 In order to state the sharp decay results,  we will   need to argue in the weighted spaces. For any $m \geq 0$,  we define the Hilbert space $L^2(m)$ as follow
 \begin{equation}
 L^2(m)= \bigg\{ f \in L^2 : \  \|f\|_{L^2(m)}=  \bigg( \int_{\rtwo} (1+ |x|^2)^m |f(x)|^2 dx\bigg)^{\f{1}{2}} < \infty \bigg\}
 \end{equation}
 One can show by means of H\"older's,  $L^2(m)(\rtwo)  \hookrightarrow L^p(\rtwo)$, whenever $1 \leq  p \leq 2$.
 We have the following  {\it a posteriori} estimate in $L^2(m)$ spaces for the solution $\tilde{\theta}$. 
 \begin{proposition}({\it A posteriori} estimates for the steady state in weighted spaces)
 	\label{t:10} 
 	
 	Assume as in Theorem \ref{theo:5}, 
 	$f\in \dot{W}^{-\al,\f{2}{\al-1}}(\rtwo): \|f\|_{\dot{W}^{-\al,\f{2}{\al-1}}}<\eps_0$. Let $1<m<3-\al$ and assume in addition $f\in W^{1-\al,2}\cap W^{-\al,2}$, $\La^{1-\al} f\in L^2(m)$. Then, $\nabla\tilde{\theta}\in L^2(m)$. 	
 \end{proposition}
 \noindent 	{\bf Remark:} In fact, there is an explicit  {\it a posteriori} estimate, see \eqref{t:20} below, for $\tilde{\theta}$ which details the particular dependence of 
 	$\|\nabla\tilde{\theta}\|_{L^2(m)}$ on various norms involving $f$ as stated above. 
 	
 	\begin{theorem}(Sharpness of the  decay estimates) 
 		\label{theo:30} 
 		
 	Let  the assumptions in Proposition \ref{t:10} stand. Assume in addition,  $v_0=\theta_0-\tilde{\theta} \in   L^\infty \cap  L^2(m)(\rtwo)$. Then, for each $\eps>0$, there exists a constant $C=C_\eps$, so that \eqref{30} has an unique global solution $v\in C[(0, \infty), L^2(m)]$, which satisfies the decay estimate
 	\begin{equation}
 	\label{t:50} 
 	\left\|v(t, \cdot)-\f{\al_0}{(1+t)^{\f{2}{\al}}} G\left(\f{\cdot}{(1+t)^{\f{1}{\al}}}\right) \right\|_{L^p}\leq \f{C_\eps}{(1+t)^{\f{m+3}{\al}-1-\f{2}{\al p}}}, 1<p\leq 2,
 	\end{equation}
 	where $\al_0(x)=\int_{\rtwo}[\theta_0(x)-\tilde{\theta}(x)] dx$. 
 	In particular,  for $\al_0\neq 0$, $0<\eps<<m-1$, and large $t$,  
 	\begin{equation}
 	\label{t:60} 
 	\|v(t,\cdot)\|_{L^p}\geq \f{|\al_0|}{2} \left\| (1+t)^{-\f{2}{\al}}  G\left(\f{\cdot}{(1+t)^{\f{1}{\al}}}\right) \right\|_{L^p} \sim (1+t)^{-\f{2}{\al}(1-\f{1}{p})}, 1<p\leq 2.
 	\end{equation}
 	\end{theorem}
 	 {\bf Remarks:} 
 	 \begin{itemize} 
 	 	\item The estimate \eqref{t:60} shows that \eqref{14} is sharp, whenever $\int_{\rtwo}[\theta_0(x)-\tilde{\theta}(x)] dx\neq 0$. 
 	\item The extra localization $v_0\in L^2(m), m>1$ guarantees $v_0\in L^1(\rtwo)$.  
 	\item It is possible to state estimates similar to \eqref{t:50}, which shows the sharpness of the decay estimates for $\|\theta(t, \cdot)-\tilde{\theta}(\cdot)\|_{L^p}$ for at least some $p>2$, but we will not do so here. 
 	 \end{itemize}

 The plan for the paper is as follows. In Section \ref{sec:2}, we first introduce some basics -  function spaces, Fourier multipliers and fractional derivatives and fractional integral operators. Next, we state and prove some properties of the Green's function of the fractional Laplacian, as well as some commutator estimates, which may be of independent interest. Lastly, we present a version of the Gronwall's lemma. In Section \ref{sec:3}, we present the details of the construction of the steady state, together with the necessary {\it a posteriori} estimates. In particular, one may find there the proofs of Theorem \ref{theo:5} and Proposition \ref{t:10}, which are mere corollaries of the more general results of this section. In Section \ref{sec:4.1}, we introduce the scaled variables for the problem.  The main advantage is that in these variables, the governing partial differential equation    is a parabolic PDE, driven by a (non-self adjoint) differential operator with purely negative spectrum, which enjoys the spectral gap property. We present a complete spectral analysis of the involved operators and the corresponding semi-group estimates, by partially relying on our previous work \cite{HS}. In Section \ref{sec:5}, we analyze the dynamics of  
 \eqref{30} in the $L^p$ setting, provided the conditions on $f$ guarantee the existence of an appropriate steady state $\tilde{\theta}$.  In particular,   the proof of Theorem \ref{theo:10} is presented. This is done by establishing appropriate $L^2$, $L^p, 2<p<\infty$ and then $L^\infty$ bounds for $v$, based on energy estimates in the unweighted spaces. These turn out to be sharp, based on the results of the next section. Importantly, it turns out that the scaled variables turn out to be an efficient medium for obtaining sharp estimates in unweighted $L^p$ spaces, even though their properties somehow suggest that they might be best  used in the weighted context. 
  Finally in Section \ref{sec:6}, we present an argument, based on energy estimates in weighted spaces $L^2(m), 1<m<3-\al$, which guarantees that the upper bounds for the decay rates are in fact optimal. This is justified by explicitly isolating the leading order term (decay wise) for the scaled variable $V$.

  \section{Preliminaries}
  \label{sec:2} 
  \subsection{Function spaces,  Fourier Transform, and multipliers} 
  \label{sec:2.1} 
  The Lebesgue $L^p$ spaces are defined by the norm   $  \|f\|_{L^p}= \bigg( \int |f(x)|^p\ dx\bigg)^{\f{1}{p}}$, while  the weak 
  $L^p$ spaces are 
  \begin{equation*}
  L^{p, \infty}= \left\{f: \|f\|_{L^{p, \infty}}=\sup_{\la>0} \bigg \{\la \ |\{x: |f(x)|> \la\} |^{\f{1}{p}} \bigg\} < \infty \right\}.
  \end{equation*}
  The Fourier transform and its inverse are taken in the form 
  $$
  \hat{f}(\xi)= \int_{\rn} f(x) e^{-i x\cdot \xi} dx, \ \ f(x) = (2\pi)^{-n} \int_{\rn} \hat{f}(\xi) e^{i x\cdot \xi} d\xi
  $$
  Consequently, since $\widehat{-\De f}(\xi) := |\xi|^2 \hat{f}(\xi)$, and as pointed out already, 
   the fractional differentiation operators are introduced via 
    $\La^a:=(-\De)^{a/2}, a>0$. Equivalently, 
  its action on the Fourier side is $\widehat{\La^a f}(\xi)= |\xi|^a \hat{f}(\xi)$.  
   In this context, recall the  Hausdorff--Young inequality which reads as follows: For 
  $p,  q, r \in (1, \infty)$ and $1+ \f{1}{p}= \f{1}{q}+ \f{1}{r}$  
  \begin{equation*}
  \|f * g\|_{L^p} \leq C_{p,q,r} \|f\|_{L^{q, \infty}} \|g\|_{L^r}.
  \end{equation*} 
  
  For an integer $n$ and $p\in (1, \infty)$,  the Sobolev spaces are the closure of the Schwartz functions in the norm $\|f\|_{W^{k, p}}= \|f\|_{L^p}+ \sum_{|\al| \leq k} \|\partial^{\al} f\|_{L^p}$,  while for a non-integer $s$  
  \begin{equation*}
  \|f\|_{W^{s, p}}= \|(1- \De)^{s/2} f\|_{L^p}\sim \|f\|_{L^p}+ \|\La^s  f\|_{L^p}.
  \end{equation*}
  We also need the homogeneous versions of it, with semi-norms  
  $\|f\|_{\dot{W}^{s, p}}= \|\La^s f\|_{L^p}$. 
  The Sobolev embedding theorem states $\|f\|_{L^p(\rn)} \leq C \|f\|_{\dot{W}^{s,q}(\rn)}$, where $1<p<q<\infty$ and 
  $n (\f{1}{p}- \f{1}{q})= s$, with the usual modification for $p=\infty$, namely 
  $\|f\|_{L^\infty(\rn)} \leq C_s \|  f\|_{W^{s,q}(\rn)}$, $s> \f{n}{p}$.
  More generally, for smooth symbols  $m$, with the property $|m(\xi)|\sim  |\xi|^s$, we have 
  \begin{equation}
  \label{m} 
  \|T_{m^{-1}} f\|_{L^p(\rn)} \leq C \|f\|_{L^q}
  \end{equation}
  where $n (\f{1}{p}- \f{1}{q})= s$ and $\widehat{T_{m^{-1}} f}(\xi)=m^{-1} (\xi) \hat{f}(\xi)$. 
  
  Finally, due to the failure of the Sobolev space $H^1(\rtwo)$ to embed in $L^\infty(\rtwo)$, we record the following modification of it:  $H^{1-\de}(\rtwo)\cap H^{1+\de}(\rtwo)\hookrightarrow L^\infty(\rtwo), \de>0$. In terms of estimates, for all $\de>0$, there exists $C_\de$, so that 
  \begin{equation}
  \label{a:40} 
  \|f\|_{L^\infty(\rtwo)}\leq C(\|f\|_{H^{1-\de}(\rtwo)}+ \|f\|_{H^{1+\de}(\rtwo)}). 
  \end{equation}

  \subsection{The fractional derivatives and anti-derivatives}
  We start by recording the following kernel representation formula for negative powers of Laplacian. This is nothing, but a fractional integral,  for $\al\in (0,2)$, 
  \begin{equation}
  \label{921}
  \La^{-\al} f(x) = c_a \int_{\rtwo} \f{f(y)}{|x-y|^{2-\al}} dy.
  \end{equation}
  Next, for positive powers, we have similar formula. More specifically, for  $ \al\in (0,2)$, 
  $$
 \La^\al f (x) = C_a p.v. \int_{\rtwo} \f{f(x)-f(y)}{|x-y|^{2+\al}} dy. 
  $$
  see Proposition 2.1, \cite{CC}). 
  Next, we have the following result, due to Chamorro and Lemari\'{e}-Rieusset, see Theorem 3.2, \cite{CL}, although for earlier version,   one may consult  Lemmas 2.4, 2.5 in \cite{CC}.  
  \begin{equation}
  \label{a:31} 
  \int_{\rn} |f(x)|^{p-2} f(x) [\La^a f](x) dx\geq  C_p \|f\|_{L^{\f{2p}{2-\al}}}^p.
  \end{equation} 
  \begin{lem}
  	\label{le:90}
  	For $p: 1\leq p<\infty$, $\al\in [0,2]$ , $n\geq 1$, 
  	\begin{equation}
  	\label{a:20} 
  	\int_{\rn} |f(x)|^{p-2} f(x) [\La^\al f](x) dx\geq 0.
  	\end{equation}
  	If in addition, $p\geq 2$, there is the stronger coercivity estimate 
  	\begin{equation}
  	\label{a:30} 
  	\int_{\rn} |f(x)|^{p-2} f(x) [\La^a f](x) dx\geq  \f{1}{p} \|\La^{\f{\al}{2}}[|f|^{\f{p}{2}-1}f]\|_{L^2(\rn)}^2.
  	\end{equation}
  \end{lem}
  	In particular, for $n=2$, by combining \eqref{a:30} with Sobolev embedding, one gets 
  	\begin{equation}
  	\label{a:34} 
  	\int_{\rtwo} |f(x)|^{p-2} f(x) [\La^a f](x) dx\geq C \|f\|_{L^{\f{2p}{2-\al}}(\rtwo)}^p,
  	\end{equation}
for some constant $C$ depending on $p,\al$. 
 We also need the following commutator estimate.

 \begin{lem} 
 	\label{L_-50} 
 	Let  $0<s<1<\si$. Then,  there is $C=C_{s,\si}$, so that 
 	\begin{equation}
 	\label{310} 
 	\|[\La^{s},|\eta|^\si]  f\|_{L^2(\rtwo)}\leq C 
 	\| |\eta|^{\si-s} f\|_{L^2(\rtwo)}.
 	\end{equation}
 \end{lem}
 We provide the slightly technical proof of Lemma \ref{L_-50} in the Appendix. We should also mention that it is roughly based on the approach for Lemma 11, \cite{HS}. 

  \subsection{The function $G$ and a variant of the Gronwall's inequality} 
  \label{sec:g}
  The function $G$ defined by $\hat{G}(\xi)=e^{-|\xi|^\al}, \xi\in \rtwo$ will be used frequently in the sequel.  Its straightforward proof can be found in \cite{HS}. 
  \begin{lem}
  	\label{le:10} 
  	For any $p \in [2, \infty]$ and $\al\in (1,2)$, 
  	\begin{equation}
  	\label{GG}
  	(1+ |\eta|^2)\ G(\eta), (1+ |\eta|^2)\nabla G(\eta) \in L_{\eta}^p
  	\end{equation} 
  	In particular, $G, \nabla G\in L^1(\rtwo)\cap L^\infty(\rtwo)$. 
  \end{lem}
  \noindent Note that $\ku_G\in L^\infty$, since 
  $$
  \|\ku_G\|_{L^\infty}\leq C \|\ku_G\|_{W^{1,4}}<\infty. 
  $$
  We have that  for $c\neq d$ and $0<a<1$, there exists $C=C(c,d,a)$, so that the following estimate holds 
  \begin{equation}
  \label{318} 
  \int_0^\tau \f{e^{-c(\tau-s)} e^{-d s}}{\min(1, |\tau-s|)^a} ds\leq C e^{-\min(c,d) \tau}.
  \end{equation}
  Moreover, we need a version of the Gronwall's inequality as follows. 
  \begin{lem}
  	\label{gro}
  	Let $\si \geq \mu>0, \ka>0$ and $a\in [0,1)$. Let $A_1,A_2, A_3$ be three positive constants so that a function $I:[0, \infty)\to \rone_+$ satisfies  $I(\tau)\leq A_1 e^{-\ga \tau}$, for some real $\ga$  and 
  	\begin{equation}
  	\label{gron} 
  	I(\tau)\leq A_2 e^{-\mu \tau} + A_3 \int_0^\tau \f{e^{-\si(\tau-s)}}{\min(1, |\tau-s|)^a} e^{-\ka s} I(s) ds.
  	\end{equation}
  	Then, there exists $C=C(a, \si, \mu, \ka, \ga, A_1, A_2, A_3)$, so that 
  	$$
  	I(\tau)\leq C   e^{- \mu \tau}.
  	$$ 
  \end{lem}
  \begin{proof}
  	We present the short proof here for completeness.  Let $\ga_n$ be so that there exists $C_n$, so that $I(\tau)\leq C_n e^{-\ga_n \tau}$ for all $\tau>0$. We will show that there exists a constant $C_{n+1}$, so that $|I(\tau)|\leq C_{n+1} e^{-\ga_{n+1} \tau}$, for $\ga_{n+1}:=\min(\mu, \f{\ka}{2}+\ga_n)$.  
  	
  	 Indeed, taking absolute values in \eqref{gron} and  plugging in the assumed estimate $I(\tau)\leq C_n e^{-\ga_n \tau}$, we obtain 
  	$$
  	|I(\tau)|\leq A_2 e^{-\mu \tau}+ A_3 C_n  \int_0^\tau \f{e^{-\si(\tau-s)}e^{-\ka s} 
  		e^{-\ga_n s} }{\min(1, |\tau-s|)^a} ds.
  	$$
  	This is of course nothing but the setup in \eqref{318}, if $\si\neq \ka+\ga_n$.  We get the estimate 
  	$$
  	|I(\tau)|\leq A_2 e^{-\mu \tau}+ D_n e^{-\min(\si, \ka+\ga_n)\tau}\leq C_{n+1} 
  	e^{-\min(\mu, \ka+\ga_n)\tau}. 
  	$$
  	Even in the case $\si=\ka+\ga_n$, via an obvious modification of the argument above, we can give up slightly in the exponents and still obtain a bound of 
  	$C_{n+1} e^{-\min(\mu, \f{\ka}{2}+\ga_n)\tau}$. 
  	
  	Thus, we have shown the bound $|I(\tau)|\leq C_{n+1} e^{-\ga_{n+1} \tau}$. The rest is just an iteration  argument, starting with $\ga_0:=\ga$, which will certainly conclude,  after a finitely many steps, since $\ka>0$,  with $\ga_N=\mu$. 
  \end{proof}
   
  \section{Construction of the steady state} 
  \label{sec:3} 
  In this section, we provide a construction of the steady state $\tilde{\theta}$. In particular, and as a corollary of the results presented herein, we show Theorem \ref{theo:5} and Proposition \ref{t:10}. The  properties of $\tilde{\theta}$ 
   will depend on the properties of the forcing term $f$.  Before we continue with the specifics, let us recast the profile   problem \eqref{20} in the more convenient form
  \begin{equation}
  \label{22} 
  \tilde{\theta} + div \La^{-\al} ( \tilde{\theta}\cdot \ku_{\tilde{\theta} }) = \La^{-\al} f,
  \end{equation}
  which was obtained using the fact that $div(\ku_{\tilde{\theta}})=0$. 
  Note that \eqref{22} (and \eqref{20}) enjoy scaling invariance. That is, if $\tilde{\theta}$ is a solution, with right-hand side $f$, then so is $\tilde{\theta}_\la(x):= \la^{\al-1} \tilde{\theta}(\la x)$, with the corresponding 
  right-hand side $f_\la(x)=\la^{1-2\al} f(\la x)$. This forces certain critical spaces in the argument, such as $\tilde{\theta}\in L^{\f{2}{\al-1}}(\rtwo), \nabla\tilde{\theta}\in L^{\f{2}{\al}}(\rtwo)$ and also $f\in \dot{W}^{-\al, \f{2}{\al-1}}$ among others. 
  As we shall need to impose smallness assumptions for our existence results, it is  well-known that these are naturally introduced in a critical space, as these norms are intrinsic (i.e. remain unchanged)  under a scaling transformation. 
  \subsection{Existence and $L^p$ properties of $\tilde{\theta}$}
  \begin{proposition}
  	\label{prop:12} 
  	Let $f\in \dot{W}^{-\al, \f{2}{\al-1}}$. Then, there exists $\eps_0>0$, so that whenever $\|f\|_{\dot{W}^{-\al, \f{2}{\al-1}}}<\eps_0$, then the equation \eqref{22} has  solution  $\tilde{\theta}\in L^{\f{2}{\al-1}}(\rtwo)$. Moreover, for some absolute constant $C$
  	$$
  	\|\tilde{\theta}\|_{L^{\f{2}{\al-1}}}\leq C \|f\|_{\dot{W}^{-\al, \f{2}{\al-1}}}. 
  	$$
  	If  for some $p>\f{2}{3-\al}$, we assume in addition $f\in W^{-\al,p}(\rtwo)$, then   $\tilde{\theta}\in L^{p}(\rtwo)$ and  
  	$$
  	\|\tilde{\theta}\|_{L^p}\leq C \|f\|_{\dot{W}^{-\al, p}}. 
  	$$
  \end{proposition}
  {\bf Remark:} We can state uniqueness results for small solutions 
  $\tilde{\theta}$ as stated above. Our dynamic results later on however, 
   provide much stronger uniqueness statements. 
  \begin{proof}
 Introduce  the operators $\ck_{h}[z]:=div \La^{-\al} ( z\cdot \ku_h)$. We will show that they map $L^{\f{2}{\al-1}}(\rtwo)$ into itself, with a norm bounded by a multiple of  $\|u\|_{L^{\f{2}{\al-1}}(\rtwo)}$. Indeed, by \eqref{m}, we have 
  	\begin{equation}
  	\label{1214} 
  		\|\ck_{h}[z]\|_{L^{\f{2}{\al-1}}(\rtwo)}\leq C \|\ku_h z\|_{L^{\f{1}{\al-1}}}\leq C 
  		\|\ku_h\|_{L^{\f{2}{\al-1}}} \|z\|_{L^{\f{2}{\al-1}}} \leq C 	\|h\|_{L^{\f{2}{\al-1}}} \|z\|_{L^{\f{2}{\al-1}}} .
  	\end{equation}
  	Thus, if $\|h\|_{L^{\f{2}{\al-1}}}<<1$, the operator $Id+\ck_h$ is invertible, via von Neumann series, with a norm $\|(Id+\ck_h)^{-1}\|_{B(L^{\f{2}{\al-1}})}\leq \f{1}{2}$. 
  	
  	With these preliminary considerations in mind, note that \eqref{22} is nothing but the functional equation 
  	$ (Id+\ck_{\tilde{\theta}})\theta = \La^{-\al} f$. 
  	Thus, we set up the iteration scheme $\tilde{\theta}_0:=\La^{-\al} f$ and for each $n\geq 1$, 
  	$
  	\tilde{\theta}_n:=(Id+\ck_{\tilde{\theta}_{n-1}}) \La^{-\al} f.
  	$
  	Clearly, this is possible, if we manage to maintain the smallness of $\|\tilde{\theta}_{n}\|_{L^{\f{2}{\al-1}}}$. This is clearly the case for  $\tilde{\theta}_0$ by assumption, since $\|\tilde{\theta}_0\|_{L^{\f{2}{\al-1}}}=\|f\|_{\dot{W}^{-\al, \f{2}{\al-1}}}<\eps_0$. 
  	 For each $n\geq 2$,  we subtract two consecutive equations to obtain\footnote{Here, we denote for conciseness $\ck_n=\ck_{\tilde{\theta}_{n}}$} 
  	 $$
  	 \tilde{\theta}_n - \tilde{\theta}_{n-1}=(\ck_{n-2}-\ck_{n-1})\tilde{\theta}_{n-1} +
  	 \ck_{n-1} (\tilde{\theta}_{n-1}-\tilde{\theta}_{n}).
  	 $$
  	Estimating as in \eqref{1214}, we obtain 
  	$$
  	\|\tilde{\theta}_n - \tilde{\theta}_{n-1}\|_{L^{\f{2}{\al-1}}}\leq C 	
  	 \|\tilde{\theta}_{n-1}\|_{L^{\f{2}{\al-1}}}
  	(\|\tilde{\theta}_{n-1} - \tilde{\theta}_{n-2}\|_{L^{\f{2}{\al-1}}}+\|\tilde{\theta}_{n} - \tilde{\theta}_{n-1}\|_{L^{\f{2}{\al-1}}}). 
  	$$
  For small enough $\eps_0$, by using an  induction arguments, we establish that $\|\tilde{\theta}_n\|<2\eps_0$ and 	 
  $$
  \|\tilde{\theta}_n - \tilde{\theta}_{n-1}\|_{L^{\f{2}{\al-1}}}\leq C\eps_0 \|\tilde{\theta}_{n-1}- \tilde{\theta}_{n-2}\|_{L^{\f{2}{\al-1}}}.
  $$
  
  This   implies that $\{\tilde{\theta}_n\}$ is a Cauchy sequence in the critical space $L^{\f{2}{\al-1}}$, which 
  means  that $\tilde{\theta}=\lim_n \tilde{\theta}_n\in L^{\f{2}{\al-1}}$ exists and it is  small. Finally, taking a limit in the $L^{\f{2}{\al-1}}$ norm  in the equation $(Id+\ck_{n-1})\tilde{\theta}_{n}=\La^{-\al} f$ implies that $(Id+\ck_{\tilde{\theta}})\tilde{\theta}=\La^{-\al} f$, which was the claim. 
  
  Now, if we assume in addition that $f\in W^{-\al, p}, p>\f{2}{3-\al}$, we obtain  
  	$$
  	\|\ck_{\tilde{\theta}}[z]\|_{L^p(\rtwo)}\leq C \|\ku_{\tilde{\theta}} z\|_{L^q}\leq C 
  	\|\ku_{\tilde{\theta}}\|_{L^{\f{2}{\al-1}}} \|z\|_{L^p} \leq C 	\|\tilde{\theta}\|_{L^{\f{2}{\al-1}}} \|z\|_{L^p},
  	$$
  	where $2\left(\f{1}{q}-\f{1}{p}\right)=\al-1$. The constraint $p>\f{2}{3-\al}$ is needed to ensure that in the above arguments $q>1$. 
  	
  \end{proof}
  Our next proposition concerns Sobolev space estimates for the steady state solution $\tilde{\theta}$  produced in Proposition \ref{prop:12}. 
  \subsection{Estimates in Sobolev spaces for $\tilde{\theta}$}
  \begin{proposition}
  	\label{prop:15} 
  		Let $f\in \dot{W}^{-\al, \f{2}{\al-1}}: \|f\|_{\dot{W}^{-\al, \f{2}{\al-1}}}<\eps_0$,  as in Proposition \ref{prop:12}. Then, there exists an absolute constant $C$, so that for each $p>1$,  the solution $\tilde{\theta}$ satisfies, 
  		\begin{equation}
  		\label{600}
  		\|\nabla \tilde{\theta}\|_{L^p(\rtwo)}\leq C \|f\|_{\dot{W}^{1-\al,p}(\rtwo)}, 
  		\end{equation}	
  		provided $f\in W^{1-\al,p}(\rtwo)$. 
  \end{proposition}
  \begin{proof}
  We set up the equation for $\nabla \tilde{\theta}$ in the form $(Id+\tilde{\ck}_{\tilde{\theta}})\nabla \theta =\nabla  \La^{-\al} f$, where the operator  $\tilde{\ck}_{h}[z]:=\nabla  \La^{-\al} ( z\cdot \ku_h)$. The operator $\tilde{\ck}$ satisfies the same bound as $\ck_{\tilde{\theta}}$ 
  $$
  \| \tilde{\ck}_{\tilde{\theta}}\|_{L^p\to L^p}\leq C \|\tilde{\theta}\|_{L^{\f{2}{\al-1}}}.
  $$
  for all $p>\f{2}{3-\al}$. Hence, the bound \eqref{600}. 
  
  Obtaining further bounds, such as  $\|\nabla \tilde{\theta}\|_{L^p(\rtwo)}$ for $1<p<\f{2}{3-\al}$ requires bootstrapping the estimates obtained for 
  $\tilde{\theta}, \nabla \tilde{\theta}$.  To this end,  we can write the equation for 
  $\nabla \tilde{\theta}$, in the form 
  \begin{equation}
  \label{700}
   \nabla \tilde{\theta} + \nabla^2 \La^{-\al} (\tilde{\theta}\cdot \ku_{\tilde{\theta}})= \nabla  \La^{-\al} f. 
  \end{equation}
  Take $L^p, 1<p<\f{2}{\al-1}$ norms in \eqref{700}. Note that since $\al<2$, we have that $\f{2}{\al-1}>\f{2}{3-\al}$ and hence this covers larger region that the needed one $1<p<\f{2}{3-\al}$.  Applying the Kato-Ponce bounds (note that $\nabla^2 \La^{-\al}$ is a pseudo-differential operator of order $2-\al$), and the Sobolev's inequality, with   $\f{1}{p}=\f{1}{r}+\f{\al-1}{2}$, 
  $$
  \|\La \tilde{\theta}\|_{L^p}\|\sim \|\nabla \tilde{\theta}\|_{L^p}\leq C \|f\|_{\dot{W}^{1-\al,p}}+ 
  C\|\La^{2-\al} \tilde{\theta}\|_{L^r} \|\tilde{\theta}\|_{L^{\f{2}{\al-1}}}\leq C \|f\|_{\dot{W}^{1-\al,p}}+ 
  C\|\La \tilde{\theta}\|_{L^p}  \|\tilde{\theta}\|_{L^{\f{2}{\al-1}}}.
  $$
 Again, the smallness obtained in Proposition \ref{prop:12},  $\|\tilde{\theta}\|_{L^{\f{2}{\al-1}}}<<1$, will allow us to hide $\|\La \tilde{\theta}\|_{L^p} $ on the left-hand side and we can obtain the bound $\|\La \tilde{\theta}\|_{L^p}\leq C \|f\|_{\dot{W}^{1-\al,p}}$. 
  \end{proof}

  \subsection{Weighted estimates for $\tilde{\theta}$}
 \begin{proposition}
 	\label{prop:20} 
 	Let $f\in \dot{W}^{-\al, \f{2}{\al-1}}: \|f\|_{\dot{W}^{-\al, \f{2}{\al-1}}}<\eps_0$,  as in Proposition \ref{prop:12}.  Let $0<\de<2-\al$ and $m=3-\al-\de$. Assume in addition $f\in W^{1-\al,2}\cap W^{-\al,2}(\rtwo)$ and $\La^{1-\al} f\in L^2(m)$. 
 	Then,  
 	\begin{equation}
 	\label{t:20} 
 		\|\nabla \theta\|_{L^2(m)}\leq  C_\de( \|\La^{1-\al} f\|_{L^2(m)}+ 
 		\|f\|_{W^{1-\al,2}}+\|f\|_{W^{-\al,2}}^2).
 	\end{equation}
 
 \end{proposition}
  \begin{proof}
  Since we need to control $\|\nabla \tilde{\theta}\|_{L^2(m)}$, we invoke Proposition \ref{prop:15} that yields control of $\|\nabla \tilde{\theta}\|_{L^2}$. It remains to control $\||x|^m \nabla \tilde{\theta}\|_{L^2(|x|>1)}$. To that end, 	
  introduce a partition of unity \\ $\sum_{k=-\infty}^\infty \chi(2^{-k} x) =1$, based on a function $\chi\in C^\infty_0$, so that $supp\chi\subset \{x: \f{1}{2}<|x|<2\}$. 
  For any function $g$, introduce the notation $g_k(x):=g(x)\chi(2^{-k}x)$. 
  
 For the rest of the argument, our goal is to  control 
  $$
  \||x|^m \nabla \tilde{\theta}\|_{L^2(|x|>1)}^2\sim \sum_{k=0}^\infty 2^{2km}  \|\nabla \tilde{\theta}_k\|_{L^2}^2.
  $$
  Multiplying   \eqref{22}  with $\chi(2^{-k}x)$ and taking $\nabla$ yields 
  $$
  \nabla \tilde{\theta}_k+ [\nabla \La^{-\al}(\ku_{\tilde{\theta}}\cdot \nabla \tilde{\theta})]_k=F_k,
  $$
  where $F_k=\nabla (\La^{-\al} f)_k=(\nabla \La^{-\al} f)_k- 2^{-k} \nabla \chi(2^{-k} \cdot) \La^{-\al} f $.

  Taking $L^2$ norms yields the relation 
  \begin{equation}
  \label{720}
  \|\nabla \tilde{\theta}_k\|_{L^2}\leq   \|F_k\|_{L^2} +   
  \|[\nabla \La^{-\al}( \ku_{\tilde{\theta}}\cdot \nabla \tilde{\theta})]_k\|_{L^2}.
  \end{equation}
  	We now estimate the non-linear term. We have that for each function $G$, there is the point-wise bound 
  	$|\nabla \La^{-\al}[G]|\leq C \La^{1-\al} |G|$. Thus, with the notation $g_{\sim k}=g_{k-2}+\ldots+g_{k+2}$, 
  	\begin{eqnarray*}
  	\|[\nabla \La^{-\al}(\ku_{\tilde{\theta}}\cdot \nabla \tilde{\theta})]_k\|_{L^2} &\leq &  \|\La^{1-\al}(\ku_{\tilde{\theta}}\cdot \nabla \tilde{\theta}_{\sim k})\|_{L^2}+ \|[\nabla \La^{-\al}(\ku_{\tilde{\theta}}\cdot \nabla \tilde{\theta}_{<k-2})]_k\|_{L^2}+\\
  	&+&  \|[\nabla \La^{-\al}(\ku_{\tilde{\theta}}\cdot \nabla \tilde{\theta}_{>k-2})]_k\|_{L^2}.
  	\end{eqnarray*}
  	For the first term, by Sobolev embedding 
  	$$
  	 \|\La^{1-\al}(\ku_{\tilde{\theta}}\cdot \nabla \tilde{\theta}_{\sim k})\|_{L^2}\leq 
  	 C \|\ku_{\tilde{\theta}}\cdot \nabla \tilde{\theta}_{\sim k}\|_{L^{\f{2}{\al}}}\leq 
  	 C \|\nabla \tilde{\theta}_{\sim k}\|_{L^2} \|\tilde{\theta}\|_{L^{\f{2}{\al-1}}}\leq 
  	 C\eps_0 \|\nabla \tilde{\theta}_{\sim k}\|_{L^2}
  	$$
For the second term and the third term, we have the point-wise estimates (recall $|x|\sim 2^k$)
\begin{eqnarray*}
|\nabla \La^{-\al}(\ku_{\tilde{\theta}}\cdot \nabla \tilde{\theta}_{< k-2})(x)| &=& 
|\nabla^2 \La^{-\al}(\ku_{\tilde{\theta}}\cdot  \tilde{\theta}_{<k-2})(x)|\leq C \int \f{1}{|x-y|^{4-\al}} |\ku_{\tilde{\theta}}(y)| |\tilde{\theta}_{<k-2}(y)| dy\leq \\
&\leq & C 2^{-k(4-\al)} \|\tilde{\theta}\|_{L^2}^2 \\
|\nabla \La^{-\al}(\ku_{\tilde{\theta}}\cdot \nabla \tilde{\theta}_{>k-2})(x)| & \leq & C 2^{-k(4-\al)} \|\tilde{\theta}\|_{L^2}^2. 
\end{eqnarray*}
Thus, 
$$
\|[\La^{1-\al}(\ku_{\tilde{\theta}}\cdot \nabla \tilde{\theta}_{< k-2})]_k\|_{L^2}+ \|[\nabla \La^{-\al}(\ku_{\tilde{\theta}}\cdot \nabla \tilde{\theta}_{<k-2})]_k\|_{L^2}\leq C 2^{-k(3-\al)} \|\tilde{\theta}\|_{L^2}^2. 
$$
  Putting everything together in \eqref{720} that 
  \begin{equation}
  \label{324} 
  \|\nabla \tilde{\theta}_k\|_{L^2}\leq   \|F_k\|_{L^2} + C 2^{-k(3-\al)} \|\tilde{\theta}\|_{L^2}^2+  C \eps_0 \|\nabla \tilde{\theta}_{\sim k}\|_{L^2}. 
  \end{equation}	
  	for all $k\geq 1$. Squaring   \eqref{324}, multiplying by $2^{2km}=2^{2k(3-\al-\de)}$ and summing in $k\geq 1$  yields the {\it a posteriori} estimate 
  	$$
  	J:=\sum_{k=1}^\infty 2^{2(3-\al-\de)k} \|\nabla \tilde{\theta}_k\|_{L^2}^2 \leq 
  \sum_{k=1}^\infty 2^{2k(3-\al-\de)} \|F_k\|_{L^2}^2 	+  \sum_{k=1}^\infty 2^{-2\de k} \|\tilde{\theta}\|_{L^2}^4+ C\eps_0 (J+\|\nabla \tilde{\theta}\|_{L^2}^2).
  	$$
  	This yields the bound, for sufficiently small $\eps_0$, 
  	$$
  	J\leq C_\de(\|\La^{1-\al} f\|_{L^2(m)}^2+ 
  	\|f\|_{W^{1-\al,2}}^2
   +\|f\|_{W^{1-\al,2}}^2+\|\tilde{\theta}\|_{L^2}^4).
  	$$
  	Thus, for $m=3-\al-\de$ and any $\de>0$, by using the bounds \eqref{600} for $p=2$, 
  	$$
  	\|\nabla \tilde{\theta}\|_{L^2(m)}\leq C_\de( \|\La^{1-\al} f\|_{L^2(m)}+ 
  	\|f\|_{W^{1-\al,2}}+\|f\|_{W^{-\al,2}}^2).
  	$$
  \end{proof}
  \section{The scaled  variables and the associated operator $\cl$} 
  \label{sec:4} 
  
  Now that we have constructed the steady state $\tilde{\theta}$ we turn our attention to the analysis of the dynamic equations. As a first step, we shall need to introduce a major technical tool of our analysis, the  scaled variables. As we have alluded to above, the scaled variable approach is  very beneficial in this context. It was pioneered in \cite{GW, GW1} for the vorticity formulation of the 2D Navier-Stokes problem  and later, it was extended in our previous work \cite{HS} to the fractional case, to establish the exact relaxation rates for very general SQG type problems. 
  
 \subsection{Scaled variables} 
 \label{sec:4.1} 
  Following \cite{HS},  we introduce the scaled variables 
  $$
  \tau= \ln(1+ t), \ \ \eta = \frac{x}{(1+ t)^{\f{1}{\al}}}.
  $$
   In the context of these new variables,  we introduce  new independent functions, 
  \begin{equation}
  \label{24} 
  v(t, x)= \frac{1}{{(1+ t)^{1- \f{1}{\al}}}} V\left(\frac{x}{(1+ t)^{\f{1}{\al}}}, \ln(1+ t)\right), \ \ 
  \widetilde{\theta}(x)= \frac{1}{{(1+ t)^{1- \f{1}{\al}}}} \Theta\left( \frac{x}{(1+ t)^{\f{1}{\al}}}\right).
  \end{equation}
  or equivalently 
  \begin{equation}
  \label{26} 
  V(\tau, \eta)=e^{\tau\left(1- \f{1}{\al}\right)} v(e^{\f{\tau}{\al}} \eta, e^{\tau}-1),  \ \ 
  \Theta(\tau, \eta)=
  e^{\tau\left(1- \f{1}{\al}\right)} \tilde{\theta}(e^{\f{\tau}{\al}} \eta). 
  \end{equation}
  We compute 
  \begin{eqnarray} \nonumber
  v_t&=& - \frac{1- \f{1}{\al}}{{(1+ t)^{2- \f{1}{\al}}}} V- \f{1}{\al} \frac{1}{{(1+ t)^{2- \f{1}{\al}}}} \f{x}{(1+ t)^{\f{1}{\al}}} \cdot  \nabla_{\eta} V+ \frac{1}{{(1+ t)^{2- \f{1}{\al}}}} V_{\tau},\\ \nonumber
  \La^{\al} v&=& \frac{1}{{(1+ t)^{2- \f{1}{\al}}}} \La^{\al} V, \hspace{1em} \La^{\al} \widetilde{\theta}= \frac{1}{{(1+ t)^{2- \f{1}{\al}}}} \La^{\al} \Theta \\ \nonumber
  \ku_v \cdot \nabla v&=& \frac{1}{{(1+ t)^{2- \f{1}{\al}}}} \ku_V \cdot \nabla  V, \hspace{1em} \ku_{\widetilde{\theta}} \cdot \nabla v= \frac{1}{{(1+ t)^{2- \f{1}{\al}}}} \ku_{\Theta} \cdot \nabla V\\
  \nonumber
  \ku_v \cdot \nabla v&=& \frac{1}{{(1+ t)^{2- \f{1}{\al}}}} \ku_V \cdot \nabla  V, \hspace{1em} \ku_{\widetilde{\theta}} \cdot \nabla v= \frac{1}{{(1+ t)^{2- \f{1}{\al}}}} \ku_{\Theta} \cdot \nabla V.
  \end{eqnarray}
In the new variables, the equation \eqref{30} transfers to
  $$
  V_{\tau}= \bigg( - \La^{\al}+ \f{1}{\al} \eta \cdot \nabla_\eta + (1- \f{1}{\al}) \bigg) V - \ku_V \cdot \nabla_\eta  V - \ku_{\Theta} \cdot \nabla_\eta  V- \ku_V \cdot \nabla_\eta  \Theta. 
  $$
  Equivalently, 
  \begin{eqnarray}\label{VV1}
  \left\{ 
  \begin{array}{l} 
  V_{\tau}= \cl V - \ku_V \cdot \nabla_\eta  V - \ku_{\Theta} \cdot \nabla_\eta  V- \ku_V \cdot \nabla_\eta  \Theta, \\ 
  V(0, \eta)= V_0(\eta),
  \end{array}  
  \right.
  \end{eqnarray} 
  where  
  \begin{equation}
  \label{lamb} 
  \cl= - \La^{\al}+ \f{1}{\al} \eta \cdot \nabla_\eta+ (1- \f{1}{\al}).
  \end{equation}
  As we shall see later, the formulation \eqref{VV1} is useful, when studying the long-time behavior of $V$ 
  (and $v$ respectively) in the Lebesgue spaces $L^p(\rtwo)$. Due to its special spectral properties of $\cl$ on the weighted spaces $L^2(2)$, the real advantage comes, when one considers $\cl$ and the associated semi-group $e^{\tau \cl}$ on the weighted space $L^2(2)$. 
  \subsection{Spectral analysis of the operator $\cl$ on $L^2(2)$} 
  Unlike the Laplacian (and the fractional Laplace operators over the appropriate domains),  which have $\si(-\La^\al)=(-\infty, 0]$,  the  operator $\cl$, with domain 
  $$
  D(\cl)=\{g\in H^\al(\rtwo)\cap L^2(2) : \cl g\in L^2(2) \} 
  $$
  pushes this spectral picture  to the left side of the imaginary axis with a gap. We take advantage of this fact, as it puts us  in a better situation that we can analysis the solutions. The following proposition, which is proved in \cite{GW1} for the case $\al=2$ and extended in \cite{HS} lists some important aspects of the spectral theory for  $\cl$. In the statements below, we quote the relevant results, as developed in our previous work \cite{HS}. 
  \begin{prop}(Proposition 2, \cite{HS})
  	\label{prop:10} 
  	Let $\mathcal{L}$ be,  as defined in \eqref{lamb}.  Then,  its spectrum on the space $L^2(2)(\rtwo)$, is described as follows 
  	\begin{enumerate}
  		\item  $\cl G=(1-\f{3}{\al}) G$ and $G\in L^2(2)(\rtwo)$, whence $G$ is an eigenfunction, corresponding to an eigenvalue $\la_0(\cl)=1-\f{3}{\al}$. 
  		\item \emph{The essential  spectrum:} Let $\mu \in \mathbb{C}$ be such that $\Re \mu  \leq  - \f{1}{\al}$ and define,
  		$\psi_{\mu} \in L^2$ such that 
  		\begin{equation}\label{con}
  		\widehat{\psi_{\mu}}(\xi)= |\xi|^{- \al \mu} e^{- |\xi|^{\al}}.
  		\end{equation}
  		Then $\psi_{\mu}$ is an  eigenfunction of the operator $\mathcal{L}$ with the corresponding eigenvalue\footnote{Note however that not all this eigenvalues are  isolated, hence they are in the essential spectrum.} $\la=1+\mu - \f{3}{\al}$.  In fact,
  		$$
  		\sigma_{ess.}(\mathcal{L}) =\bigg\{ \la \in \mathbb{C}:  \Re \la  \leq  1-  \f{4}{\al}  \bigg \}.
  		$$
  		In particular, 
  		$$
  		\si(\cl)=\{1-\f{3}{\al}\}\cup \bigg\{ \la \in \mathbb{C}:  \Re \la  \leq  1-  \f{4}{\al}  \bigg \},
  		$$
  		where $\la_0(\cl)=1-\f{3}{\al}$ is a simple eigenvalue, with an eigenfunction $G$. 
  		\item 	The operator $\cl$ defines a $C_0$ semi-group, $e^{\tau \cl}$ on $L^2(2)$.  In fact, we have the following  formulas  for its action 	
  		\begin{eqnarray}
  		\label{semi1}
  		\widehat{(e^{\tau \cl} f)}(\xi)&=& e^{(1-\f{3}{\al}) \tau} e^{- a(\tau) |\xi|^{\al}} \widehat{f}(e^{-\f{ \tau}{\al}} \xi),  \\ \label{semi2}
  		(e^{\tau \cl}f) (\eta)&=& \f{e^{(1-\f{1}{\al})\tau}}{a(\tau)^{\f{2}{\al}}} \int_{\rtwo} G \left(\f{\eta- p}{a(\tau)^{\f{1}{\al}}}\right) f(e^{\f{\tau}{\al}}p ) d p,
  		\end{eqnarray} 
  		where $a(\tau)= 1- e^{- \tau}$.
  		\item There is the commutation formula 
  		\begin{equation}
  		\label{28} 
  		e^{\tau \cl}\nabla = e^{-\f{\tau}{\al}} \nabla e^{\tau \cl}.
  		\end{equation}
  	\end{enumerate}
  \end{prop}
  The next lemma presents an estimate for the bounds of the semi-group $e^{\tau \cl}$ on $L^2(2)$. Note 
  the requirement $\hat{f}(0)=0$, which is necessary for the bounds to hold. 
  \begin{lem}(Proposition 3, \cite{HS})
  	\label{lem1} 
  	
  	Let $f \in L^2(2)$, $\hat{f}(0)= 0$ and $\ga=(\ga_1, \ga_2) \in {\mathbf N}^2, |\ga|=0,1$ and $0<\eps <<1$. Then there exists $C=C_\eps>0$,  such that for any $\tau> 0$, 
  	\begin{equation}  
  	\label{18}
  	\|\nabla^{\ga} ( e^{\tau \cl} f)\|_{L^2(2)} \leq C \f{e^{\left(1-\f{4}{\al}+\eps \right) \tau}}{a(\tau)^{\f{|\ga|}{\al}}} \|f\|_{L^2(2)},
  	\end{equation}
  	or 
  	\begin{equation}  
  	\label{1888}
  	\|\nabla^{\ga} ( e^{\tau \cl} f)\|_{L^2(2)} \leq C_\eps \|f\|_{L^2(2)} \left\{ 
  	\begin{array}{l}
  	\f{1}{\tau^{\f{|\ga|}{\al}}},  \ \ \ \ \ \ \ \  \tau \leq 1 \\
  	e^{\left(1-\f{4}{\al}+\eps \right)\tau }, \ \ \ \ \tau > 1 
  	\end{array}  
  	\right.\cdot
  	\end{equation}
  \end{lem} 
  Due to the formula \eqref{28}, we have the bound 
  \begin{equation}  
  \label{19}
  \|  ( e^{\tau \cl} \nabla  f)\|_{L^2(2)} \leq C_\eps \f{e^{\left(1-\f{5}{\al}+\eps \right) \tau}}{\min(1, \tau)^{\f{1}{\al}}} \|f\|_{L^2(2)}.
  \end{equation}
  By taking advantage of the representation formula \eqref{semi2}, Lemma \ref{le:10} and again the commutation formula \eqref{28},   one derives the action of $e^{\tau\cl}$ as an element of $B(L^p, L^q)$. 
  \begin{lem}
  	\label{aa}
  	Let $\al > 0$ and $1 \leq p \leq q \leq \infty$.  Then for any $\tau > 0$ 
  	\begin{eqnarray} \label{2.7}
  	\|e^{\tau \cl} f\|_{L^q} &\leq & C \f{e^{(1-\f{1}{\al}- \f{2}{\al p})\tau}}{(a(\tau))^{\f{2}{\al} (\f{1}{p}- \f{1}{q})}} \|f\|_{L^p},  \\ 
  	\label{2.8}
  	\|e^{\tau \cl} \nabla f\|_{L^q} &\leq & C \f{e^{(1-\f{2}{\al}- \f{2}{\al p})\tau}}{(a(\tau))^{\f{2}{\al} (\f{1}{2}+ \f{1}{p}- \f{1}{q})}} \|f\|_{L^p}. 
  	\end{eqnarray}
  \end{lem}
  Next, we discuss the spectral projection along the first eigenvalue and related operators. This is discussed in great detail in Section 3.4, \cite{HS}, so we just state the main results. 
  \begin{prop}
  	\label{prop:98} 
  	The Riesz projection onto the eigenvalue $\la_0(\cl)=1-\f{3}{\al}$ is given by the formula 
  	$$
  	\cp_0 f=\left(\int_{\rtwo} f(\eta) d\eta\right) G=\dpr{f}{1} G. 
  	$$
  	The operator $\cq_0:=Id-\cp_0$ is a projection  over the rest of the spectrum  
  	$\si_{ess.}(\cl)=\{\la: \Re \la\leq 1-\f{4}{\al}\}$.   Moreover, for all $\eps>0$, there are  the estimate 
  	\begin{eqnarray}
  	\label{300} 
  	& & 	\|\nabla^{\ga} ( e^{\tau \cl} \cq_0 f)\|_{L^2(2)} \leq C \f{e^{\left(1-\f{4}{\al}+\eps \right) \tau}}{a(\tau)^{\f{|\ga|}{\al}}} \|f\|_{L^2(2)}, \\
  	& & 
  	\label{313} 
  	\|  ( e^{\tau \cl} \cq_0 \nabla f)\|_{L^2(2)} \leq C \f{e^{\left(1-\f{5}{\al}+\eps \right) \tau}}{a(\tau)^{\f{1}{\al}}} \|f\|_{L^2(2)}. 
  	\end{eqnarray}
  \end{prop}
  
  \subsection{Spectral analysis on $L^2(m), 1<m<2$}
  \begin{corollary}
  	\label{cor:26} 
  	Let $\mathcal{L}$ be as defined in \eqref{lamb}.  Then,  its spectrum on the space $L^2(m)(\rtwo)$, $1<m <2 $ is described as follows 
  	\begin{enumerate}
  		\item $\la_0(\cl)=1-\f{3}{\al}$  is simple eigenvalue,   with an eigenfunction $G$.
  		\item 
  		$$
  		\si(\cl)\setminus \{1-\f{3}{\al}\} \subseteq \bigg\{ \la \in \mathbb{C}:  \Re \la  \leq  
  		1- \f{m+2}{\al}  \bigg \},
  		$$
  	\end{enumerate}
  	In addition, there are the bounds, for $|\ga|=0,1$, $f\in L^2(m), \hat{f}(0)=0$, we have 
  	\begin{equation}  
  	\label{16}
  	\|\nabla^{\ga} ( e^{\tau \cl} f)\|_{L^2(m)} \leq C \f{e^{\left(1-\f{m+2}{\al}+\eps \right) \tau}}{a(\tau)^{\f{|\ga|}{\al}}} \|f\|_{L^2(m)},
  	\end{equation}
  	and 
  	\begin{equation}  
  	\label{32}
  	\|  ( e^{\tau \cl} \nabla  f)\|_{L^2(m)} \leq C_\eps \f{e^{\left(1-\f{m+3}{\al}+\eps \right) \tau}}{\min(1, \tau)^{\f{1}{\al}}} \|f\|_{L^2(m)},
  	\end{equation}
  	
  \end{corollary}
  
  \begin{proof}
  	First, observe that since $L^2(m)(\rtwo)\subset L^1(\rtwo)$, the operators $\cp_0, \cq_0$ are well-defined. The eigenvalue $1-\f{3}{\al}$ is valid by inspection. 
  	The formula for the spectrum follows in an identical way as in Proposition \ref{prop:10}, once we establish the estimates \eqref{16} and \eqref{32}. 
  	Their  proofs are obtained by interpolation of the corresponding $L^2\to L^2$ bounds, found in \eqref{2.7} and \eqref{2.8} and the $L^2(2)\to L^2(2)$ bounds, in \eqref{18}, \eqref{19}. 
  \end{proof}

  \section{ A priori    estimates in $L^p$ spaces} 
  \label{sec:5} 
  Let us first record for future reference some expressions for $\|\Theta(\tau, \cdot)\|_{L^p}, \|\nabla \Theta(\tau, \cdot)\|_{L^p}$ and \\ $\|\Theta(\tau, \cdot)\|_{L^2(2)}$ 
  \begin{eqnarray}
  \label{40} 
  	\|\Theta\|_{L^p} &=& e^{(1- \f{1}{\al}- \f{2}{\al p}) \tau} \|\tilde{\theta}\|_{L^p}, 	\|\nabla \Theta\|_{L^p} = 
  	e^{(1 - \f{2}{\al p}) \tau} \|\nabla \tilde{\theta}\|_{L^p},  \\
  	\label{42} 
  		\|\Theta\|_{L^2(2)} &\leq &  C  e^{(1- \f{2}{\al}) \tau} \|\tilde{\theta}\|_{L^2(2)}.
  \end{eqnarray}
 Clearly, these formulas follow from the relation $\Theta(\tau, \eta)=
 e^{\tau\left(1- \f{1}{\al}\right)} \tilde{\theta}(e^{\f{\tau}{\al}} \eta)$. 
 We start our {\it a priori estimates with $V$}, more precisely for $\|V(\tau, \cdot)\|_{L^p(\rtwo)}, 2\leq p<\infty$. 
 \subsection{$L^p, 2\leq p<\infty$ bounds} 
 \begin{lem}
 	\label{lem5}
 	 There exists $\eps_0>0$, so that whenever $\|\nabla \tilde{\theta}\|_{L^{\f{2}{\al}}(\rtwo)}<\eps_0$, then the solution $V$ of \eqref{VV1} satisfies 
 	\begin{equation}
 	\label{60} 
 	\|V(\tau, \cdot)\|_{L^2(\rtwo)}\leq \|V_0(\cdot)\|_{L^2(\rtwo)} e^{\tau\left(1-\f{2}{\al}\right)}.
 	\end{equation}
 	Moreover, for every $2\leq p<\infty$, there exists $\eps_0=\eps_0(p)$, so that whenever 
 	$\|\nabla \tilde{\theta}\|_{L^{\f{2}{\al}}}<\eps_0(p)$, then for all $q: 2\leq q\leq p$, we have 
 		\begin{equation}
 		\label{65} 
 		\|V(\tau, \cdot)\|_{L^q}\leq C_q \|V_0(\cdot)\|_{L^q\cap L^2} e^{\tau\left(1-\f{2}{\al}\right)}.
 		\end{equation}
 \end{lem}
{\bf Remarks:} 
\begin{itemize} 
	\item Note that since $\|\nabla \tilde{\theta}\|_{L^{\f{2}{\al}}}=\|\nabla \Theta\|_{L^{\f{2}{\al}}}$, Proposition \ref{prop:15} ensures that the assumptions are satisfied, whenever $\|f\|_{\dot{W}^{1-\al, \f{2}{\al}}}<<1$. 
	\item While we do require  the smallness of $\|\nabla \tilde{\theta}\|_{L^{\f{2}{\al}}}=\|\nabla \Theta\|_{L^{\f{2}{\al}}}$, it is important to point out that 
we {\it do not require
	$ \|V_0(\cdot)\|_{L^q}$ to be small}. 
	\item From our proof, we can only show \eqref{65}, under the assumption 
	 $\lim_{p\to \infty} \eps_0(p)=0$. In other words, for each $p>2$, we need to impose that $\|\nabla \tilde{\theta}\|_{L^{\f{2}{\al}}}$ is progressively smaller and smaller, before we can claim \eqref{65}. This may or may not be optimal, but this is why we cannot claim that there is an universal $\eps_0$, which would guarantee \eqref{65} for all  $1<q<\infty$. 
\end{itemize}
\begin{proof}
As pointed out, $\|\nabla \Theta\|_{L^{\f{2}{\al}}}=\|\nabla \tilde{\theta}\|_{L^{\f{2}{\al}}}<<1$. Let $p>1$ and  take  the dot product of the equation \eqref{VV1} with $|V|^{p-2} V$. Using \eqref{a:34}, we have 
$$
\dpr{\La^\al V}{|V|^{p-2} V}\geq C_p \|V\|^p_{L^{\f{2 p}{2- \al}}}. 
$$
Supplementing this estimate with integration by parts implies 
$$
 \f{1}{p} \partial_{\tau} \|V\|^p_{L^p}+ (\f{2}{\al p}+ \f{1}{\al }- 1) \|V\|^p_{L^p}+ C_p \|V\|^p_{L^{\f{2 p}{2- \al}}} =    \bigg| \int (\ku_V \cdot \nabla \Theta)  |V|^{p- 2} V d\eta \bigg|\leq 
C \|\nabla \Theta\|_{L^{\f{2}{\al}}} \|V\|^p_{L^{\f{2 p}{2- \al}}}. 
$$
	where we have used the H\"older's inequality and $\|\ku_V\|_{L^{\f{2 p}{2- \al}}} \leq C \|V\|_{L^{\f{2 p}{2- \al}}}$. 
	
	Specializing first to $p=2$ and taking into account the smallness  $\|\nabla \Theta\|_{L^{\f{2}{\al}}}<<1$, we obtain 
\begin{equation}
\label{98} 
	\partial_{\tau} \|V\|^2_{L^2}+ 2(\f{2}{\al } - 1) \|V\|^2_{L^2}+ C \|V\|^2_{L^{\f{4}{2- \al}}}\leq \f{C}{2}   \|V\|^2_{L^{\f{4}{2- \al}}},
\end{equation}
	in particular $\partial_{\tau} \|V\|^2_{L^2}+ 2(\f{2}{\al } - 1) \|V\|^2_{L^2}\leq 0$. Resolving this differential inequality implies \eqref{60}. 
	
	For the general case, and by taking into account \eqref{60}, we can perform similar arguments. A point of notable difference is that since for sufficiently large $p$ (and we do need \eqref{65} for arbitrarily large $p$!), one may have that 
	$(\f{2}{\al p}+ \f{1}{\al }- 1) <0$, which is problematic. In order to fix this issue, we add $C \|V\|^p_{L^p}, C>>1$ to the energy estimate. We obtain 
	\begin{eqnarray*}
		& & \f{1}{p} \partial_{\tau} \|V\|^p_{L^p}+ (C+\f{2}{\al p}+ \f{1}{\al }- 1) \|V\|^p_{L^p}+ C_p \|V\|^p_{L^{\f{2 p}{2- \al}}}  
		\leq 
		C \|\nabla \Theta\|_{L^{\f{2}{\al}}} \|V\|^p_{L^{\f{2 p}{2- \al}}}+ C \|V\|^p_{L^p}\leq \\
		&\leq & 	\f{C_p}{2}   \|V\|^p_{L^{\f{2 p}{2- \al}}} + \f{C_p}{2}  \|V\|^p_{L^{\f{2 p}{2- \al}}}+ D_p  \|V\|^p_{L^2},
	\end{eqnarray*}	
	where in the last inequality we have used the smallness of $\|\nabla \Theta\|_{L^{\f{2}{\al}}}$ and the Gagliardo - Nirenberg's estimate  
	$\|V\|^p_{L^p}\leq \f{C_p}{2}  \|V\|^p_{L^{\f{2 p}{2- \al}}}+ D_p  \|V\|^p_{L^2}$. As a consequence, since $\|V\|_{L^2}\leq C e^{(1-\f{2}{\al})\tau}$, for some constant $R_p$, 
\begin{equation}
\label{70} 
	\partial_{\tau} \|V\|^p_{L^p}+p(C+\f{2}{\al p}+ \f{1}{\al }- 1) \|V\|^p_{L^p} \leq p D_p \|V\|_{L^2}^p\leq 	R_p  
	e^{p \tau\left(1-\f{2}{\al}\right)}.
\end{equation}
Resolving the differential inequality \eqref{70} leads us to 
$$ 
\|V(\tau, \cdot)\|^p_{L^p}\leq \|V_0\|^p_{L^p} e^{-p\tau (C+\f{2}{\al p}+ \f{1}{\al }- 1) } + R_p 
\int_0^\tau 
e^{-p(\tau-s) (C+\f{2}{\al p}+ \f{1}{\al }- 1) }  
e^{p s\left(1-\f{2}{\al}\right)} ds.
$$
Applying \eqref{318},  with comfortably large $C$ yields the bound \eqref{65} for $q=p$. 
	
	Let us finish with a few words regarding an extension of this to all $2\leq q\leq p$, as announced in \eqref{65}, which also elucidates the reason one cannot possibly extend this to all $p<\infty$.  If one traces the argument above, we see that since $C_p\sim p^{-1}$, one needs smallness assumption in the form $\|\nabla \Theta\|_{L^{\f{2}{\al}}}\leq Cp^{-1}$, which clearly cannot hold for all $p<\infty$. On the other hand, for each fixed $p<\infty$, we can find $\eps_p\sim p^{-1}$, so that $\|\nabla \Theta\|_{L^{\f{2}{\al}}}\leq C q^{-1}$ for all $2\leq q\leq p$, which in turn  implies \eqref{65} by the above arguments. 
\end{proof}
Using the formulas \eqref{26}, we arrive at the following corollary of Lemma \ref{lem5}.  
\begin{cor}
	\label{cor:1} 
	Let $2\leq p<\infty$, and $\tilde{\theta}: \|\nabla \tilde{\theta}\|_{L^{\f{2}{\al}}}<\eps_0(p)$.  Then, for every initial data $v_0\in L^p\cap L^2$ of the IVP \eqref{30}, we have  the decay bound
	\begin{equation}
	\label{90} 
	\|v(t, \cdot)\|_{L^q(\rtwo)}\leq C \|v_0\|_{L^q}  (1+t)^{\f{\f{2}{q}-1}{\al}}, 2\leq q\leq p.
	\end{equation}
\end{cor}
Note that in the estimate \eqref{90}, one does not get {\it any} decay for the case $q=2$. This is slightly worse than the corresponding results in \cite{Dai}, where it is shown that $\lim_{t\to \infty} \|v(t, \cdot)\|_{L^2(\rtwo)}=0$. On the other hand, even in the case of zero forcing, $f=0$,   Niche and  Schonbek, \cite{NS} have established that the rate of decay for $\|v(t, \cdot)\|_{L^2(\rtwo)}$ could be arbitrarily  slow, in particular {\it one should not be able to get any power rate for the case $q=2$}.

Next, we present some {\it a posteriori estimates} for $\|V(\tau, \cdot)\|_{L^p}$ in the cases $1<p<2$. 
\subsection{$L^p, 1<p<2$ bounds}
In this section, we show that the estimates obtained in Lemma \ref{lem5} could be improved substantially, if one assumes that $V_0\in L^1(\rtwo)$, or even $V_0\in L^p(\rtwo), 1<p<2$. We have the following 
\begin{lem}
	\label{lem7} 
	Assume that the smallness condition $\|\nabla \tilde{\theta}\|_{L^{\f{2}{\al}}(\rtwo)}<\eps_0(\f{2}{\al-1})$ and $\nabla\tilde{\theta}\in L^{\f{4}{2+\al}}(\rtwo)$. Let 
	$V_0\in L^1(\rtwo)\cap L^\infty(\rtwo)$. Then, 
	\begin{equation}
	\label{115} 
	\|V(\tau, \cdot)\|_{L^1\cap L^2}\leq C e^{\tau(1-\f{3}{\al})}. 
	\end{equation}
	Moreover, for every $2<p<\infty$, there exists $\eps_0=\eps_0(p)$, so that whenever $\tilde{\theta}$ satisfies the smallness condition $ \|\nabla \tilde{\theta}\|_{L^{\f{2}{\al}}(\rtwo)}<\eps_0(p)$
		\begin{equation}
		\label{125} 
		\|V(\tau, \cdot)\|_{L^q}\leq C_p e^{\tau(1-\f{3}{\al})}, 2<q<p.
		\end{equation}
\end{lem}
\noindent {\bf Remarks:} 
\begin{itemize}
	\item According to Proposition \ref{prop:15}, the conditions on $\tilde{\theta}$ are ensured by $f\in W^{1-\al, \f{4}{2+\al}}$ 
	and \\ $\|f\|_{\dot{W}^{-\al , \f{2}{\al-1}}}<<1$. 
	\item We point out again,  that we do not require smallness of $\|V_0\|_{L^1(\rtwo)\cap L^\infty(\rtwo)}$. 
\end{itemize}
\begin{proof}
	The proof is a bootstrap of the bounds \eqref{60} and \eqref{65}. In order to proceed with the steps, assume that we have the bound $\|V(\tau, \cdot)\|_{L^2}\leq C e^{s_n\tau}$, with $s_n<1-\f{3}{\al}$. Clearly, we start with \eqref{60}, which is $s_0=1-\f{2}{\al}$. We apply the energy estimate \eqref{70} to it, so we obtain 
	$\|V(\tau, \cdot)\|_{L^p}\leq C e^{s_n \tau}$ as well. 
	
	For $p>1$, take dot product of \eqref{VV1} with $|V|^{p-2} V$. Applying the same estimates as in the beginning of the proof of Lemma \ref{lem5}, we obtain 
	\begin{equation}
	\label{100}
	\f{1}{p} \partial_{\tau} \|V\|_{L^p}^p+ ( \f{2}{\al p}+\f{1}{\al }- 1) \|V\|_{L^p}^p \leq     \bigg| \int (\ku_V \cdot \nabla \Theta)  
	|V|^{p-2} V d\eta \bigg|. 
	\end{equation}	
		We estimate the right hand-side, for some large $q$ (to be determined momentarily),   by \\ 
		$C\|V(\tau, \cdot)\|_{L^{ p q}}^p \|\nabla \Theta\|_{L^{q'}}\leq 
		C e^{p s_n \tau} e^{\tau(1-\f{2}{q' \al})}$, since $\|\nabla \Theta\|_{L^{q'}}= e^{\tau(1-\f{2}{q' \al})}
		\|\nabla \tilde{\theta}\|_{L^{q'}}$. 	Plugging this estimate back in \eqref{100} yields 
		\begin{equation}
		\label{t:27} 
			\partial_{\tau} \|V\|_{L^p}^p+p(\f{2}{\al p}+\f{1}{\al }- 1) \|V\|_{L^p}^p\leq C 
			e^{\tau(p s_n  +1-\f{2}{q' \al})}.
		\end{equation}
	Choosing $p=1$  and $q=\f{4}{2+\al}$, so that $1-\f{2}{q'\al}=\f{1-\f{2}{\al}}{2}$, and resolving the differential inequality \eqref{t:27}, we obtain the bound 
	\begin{equation}
	\label{t:37} 
	\|V(\tau, \cdot)\|_{L^1}\leq C e^{\tau\max(1-\f{3}{\al}, 
		s_n+(\f{1}{2} - \f{1}{\al}))}.
	\end{equation}
 	In order to establish \eqref{115}, it remains to obtain the better estimate for $	\|V(\tau, \cdot)\|_{L^2}$. We proceed starting with  \eqref{98}, by adding $2C \|V\|_{L^2}^2$ for large $C$. We have  by the Gagliardo-Nirenberg's
 	\begin{eqnarray*}
 		\partial_{\tau} \|V\|^2_{L^2}+ 2(\f{2}{\al } - 1+C) \|V\|^2_{L^2}+\f{C}{2} \|V\|_{L^{\f{4}{2-\al}}}^2 
 		&\leq &  2C \|V\|^2_{L^2}\leq D \|V\|_{L^{\f{4}{2-\al}}}^{\f{4}{2+\al}}  \|V\|_{L^1}^{\f{2 \al}{2+\al}} \\
 		&\leq & 
 		\f{C}{2} \|V\|_{L^{\f{4}{2-\al}}}^2 + C_\al \|V\|_{L^1}^2.
 	\end{eqnarray*}
	Simplifying and using the bound \eqref{t:37}, leads to  
	\begin{equation}
	\label{116} 
	\partial_{\tau} \|V\|^2_{L^2}+ 2(\f{2}{\al } - 1+C) \|V\|^2_{L^2}\leq 
	C e^{2 \tau\max(1-\f{3}{\al}, s_n+(\f{1}{2} - \f{1}{\al}))}
	\end{equation}
	Resolving this last differential inequality, by making sure that $C>\f{1}{\al}$,  leads to 
		\begin{equation}
		\label{118} 
 \|V(\tau, \cdot)\|_{L^2}\leq C e^{\tau\max(1-\f{3}{\al}, 
 	s_n+(\f{1}{2} - \f{1}{\al}))}.
		\end{equation}
If $	s_n+(1-\f{2}{q' \al})\leq 1-\f{3}{\al}$, then we are done, as we have proved \eqref{115}. Otherwise, we have shown 
\begin{equation}
\label{119} 
\|V(\tau, \cdot)\|_{L^2}\leq C e^{s_{n+1}\tau},
\end{equation}
	where $s_{n+1}=s_n+ \left(\f{1}{2}-\f{1}{\al}\right)$, by the choice of $q$. Iterating the bounds $\|V(\tau, \cdot)\|_{L^2}\leq C e^{s_{n}\tau}$, whenever $s_n\leq 1-\f{3}{\al}$, with $s_{n+1}= s_n+ \left(\f{1}{2}-\f{1}{\al}\right)$ will lead to the bound \eqref{115} in finitely many steps. 
	
	Regarding the extension to \eqref{125}, we use the bound leading to \eqref{70}, which reads\footnote{note that its derivation relies on the fact that 
		$\|\nabla \Theta\|_{L^\f{2}{\al}}=\|\nabla \tilde{\theta}\|_{L^\f{2}{\al}} <\eps_0(p)$.}
	\begin{equation}
	\label{t:62} 
	\p_\tau \|V(\tau, \cdot)\|_{L^q}^q+q(C+\f{2}{\al q}+\f{1}{\al}-1) \|V(\tau, \cdot)\|_{L^q}^q\leq D_p \|V(\tau, \cdot)\|_{L^2}^q. 
	\end{equation}
	for all $2<q<p$. 
 Now, we just insert the bound \eqref{115} on the right hand side of \eqref{t:62} and we solve the resulting differential inequality
 $$
 \p_\tau \|V(\tau, \cdot)\|_{L^q}^q+q(C+\f{2}{\al q}+\f{1}{\al}-1) \|V(\tau, \cdot)\|_{L^q}^q\leq D_p e^{q(1-\f{3}{\al})\tau}. 
 $$
For a comfortably large $C$, which we can select at our will, this results in \eqref{125}. 
\end{proof}
As an obvious corollary, we have 
\begin{cor}
	\label{cor:12} 
	Let $p>2$ and $v_0\in L^1(\rtwo)\cap L^\infty(\rtwo)$, $f\in W^{1-\al, \f{4}{2+\al}}$.  Then, there exists $\eps_0=\eps_0(\al, p)$, so that whenever  $\tilde{\theta}: \|\nabla \tilde{\theta}\|_{L^{\f{2}{\al}}}<\eps_0$, we have the bounds 
\begin{equation}
\label{120} 
\|v(t, \cdot)\|_{L^p}\leq C(1+t)^{-\f{2}{\al} \left(1-\f{1}{p}\right)}. 
\end{equation}
for some constant $C=C(p, \al, \theta_0,f)$. 
\end{cor}
\subsection{$L^\infty$ bounds} 
Our next task is to establish an exponential  decay for $\|\ku_V(\tau, \cdot)\|_{L^\infty}$, as our subsequent  arguments demand it. This is not so straightforward for at least two reasons - first, by the failure of the Riesz transform to act boundedly on $L^\infty$, we may not directly pass from $\|\ku_V\|_{L^\infty}$ to $\|V\|_{L^\infty}$, and secondly -  one does not 
have ready-to-use estimate for $\|V\|_{L^\infty}$, see Lemma \ref{lem5} above. Instead, we use the Sobolev embedding, along with the boundedness of the Riesz transforms on $W^{s,p}$ spaces as follows 
\begin{equation}
\label{87} 
\|\ku_V\|_{L^\infty}\leq C_{p,s} \|\ku_V\|_{W^{s,p}(\rtwo)}\leq C_{p,s} \|V\|_{W^{s,p}(\rtwo)}, 
\end{equation} 
as soon as $s>\f{2}{p}$. Incidentally, \eqref{87}also provides bounds for $\|V\|_{L^\infty}$,  as the same chain of inequalities  applies for it as well. Thus, our goal is to find bounds for $\|V(\tau, \cdot)\|_{W^{s,p}(\rtwo)}$. 
Unfortunately, such bounds, especially one with exponential decay in $\tau$ are not easy to come by.
 On the other hand, it suffice to find inefficient ones, which then can be used in a Gagliardo-Nirenberg's fashion, together with \eqref{65},  to produce the required exponential decay for appropriate $\|V\|_{W^{s,p}(\rtwo)}, s>\f{2}{p}$. To that end, it suffices  to estimate $\|\nabla \theta(t, \cdot)\|_{L^p}$. 
 \begin{lem}
 	\label{le:4}
 	Let  $\theta_0, \nabla \theta_0 \in L^1(\rtwo)\cap L^\infty(\rtwo)$, with 
 	$\|\nabla \tilde{\theta}\|_{L^{\f{2}{\al}}}<\eps_0(\f{3\al}{\al-1})$. Let also $2<p<\infty$ and $\nabla f \in L^{\f{2p}{2+\al p-\al}}$.   Then, there exists $A_{p, \al}$, so that 
 	\begin{equation}
 	\label{75} 
 	\|\nabla \theta(t, \cdot)\|_{L^p}\leq C_p(\theta_0)   (1+t)^{A_{p,\al}}. 
 	\end{equation} 
 \end{lem}
 {\bf Note:} Here, the constant $A_{p, \al}$ is fairly large, which makes \eqref{75} pretty ineffective. We remind ourselves however  that this estimate is only very  preliminary and it will be bootstrapped later on.  
\begin{proof}
	We differentiate the original equation \eqref{10}, we use $\p$ to denote any of $\p_j, j=1,2$. 
\begin{equation}
\label{93} 
	\p_t \p\theta+ \La^\al \p\theta+ \p \ku_\theta \cdot \nabla \theta+ \ku_\theta \cdot \nabla \p \theta= \p f. 
\end{equation}
Our first step is to control $\|\nabla \theta\|_{L^2}$. To this end, take dot product with $\p \theta$. After adding  in $j=1,2$ provides the bound 
	\begin{equation}
	\label{77} 
	\p_t \|\nabla \theta\|_{L^2}^2 + C \|\La^{1+\f{\al}{2}} \theta\|_{L^2}^2 \leq B  \|\La^{1+\f{\al}{2}} \theta\|_{L^2} 
	\|\La^{1-\f{\al}{2}}  f\|_{L^2} + C\|\nabla \theta\|_{L^3}^3. 
	\end{equation} 
	Clearly, $B  \|\La^{1+\f{\al}{2}} \theta\|_{L^2} 
	\|\La^{1-\f{\al}{2}}  f\|_{L^2} \leq \f{C_p}{2}\|\La^{1+\f{\al}{2}} \theta\|_{L^2}^2+D_p \|\La^{1-\f{\al}{2}}  f\|_{L^2}^2$. 	
	Furthermore, the Gagliardo-Nirenberg's  and Young's inequalities imply 
	$$
	\|\nabla \theta\|_{L^3}^3\leq C 	\|\La \theta\|_{L^3}^3\leq C  \|\La^{1+\f{\al}{2}} \theta\|_{L^2}^{\f{6}{2+\al}} \|\theta\|_{L^{\f{3\al}{\al-1}}}^{\f{3\al}{2+\al}}\leq \f{C_p}{2}\|\La^{1+\f{\al}{2}} \theta\|_{L^2}^2+ B_p\|\theta\|_{L^{\f{3\al}{\al-1}}}^{\f{3\al}{\al-1}}. 
	$$
	Putting it all together implies 
\begin{equation}
\label{95} 
	\p_t \|\nabla \theta\|_{L^2}^2\leq B_p\|\theta\|_{L^{\f{3\al}{\al-1}}}^{\f{3\al}{\al-1}}.
\end{equation}
	Keeping in mind that $\theta=\tilde{\theta}+v$ and the decay bound\footnote{which applies since $\tilde{\theta}$ is small enough as in the Corollary \ref{cor:1}}   \eqref{90} for $v$, we conclude 
	$\|\theta(t, \cdot)\|_{L^{\f{3\al}{\al-1}}}^{\f{3\al}{\al-1}}\leq C$ and so, \eqref{95} implies, after integration in time,  
	$\|\nabla \theta(t, \cdot)\|_{L^2}\leq C (1+t)^{1/2}$.  	This estimate serves as
	 a preliminary step towards  controlling 
  $	\|\nabla \theta(t, \cdot)\|_{L^p}$. 
	
	We now proceed to estimate $	\|\nabla \theta(t, \cdot)\|_{L^p}$. 
	Taking dot product of \eqref{93}  with $|\p \theta|^{p-2} \p \theta$ and adding in $j=1,2$, we obtain, in a manner similar to the energy estimate above 
\begin{equation}
\label{80} 
  \p_t \|\nabla \theta\|_{L^p}^p + C_p \|\nabla \theta\|_{L^\f{2p}{2-\al}}^p\leq B_p \|\nabla \theta\|_{L^\f{2p}{2-\al}}^{p-1} 
 	\|\nabla f\|_{L^{\f{2p}{2+\al p-\al}}} + C\|\nabla \theta\|_{L^{p+1}}^{p+1} 
\end{equation}
By the Young's inequality $B_p\|\nabla \theta\|_{L^\f{2p}{2-\al}}^{p-1} 
\|\nabla f\|_{L^{\f{2p}{2+\al p-\al}}} \leq \f{C_p}{2}  \|\nabla \theta\|_{L^\f{2p}{2-\al}}^p+ D_p	\|\nabla f\|_{L^{\f{2p}{2+\al p-\al}}} ^p$, so that the term $\f{C_p}{2}  \|\nabla \theta\|_{L^\f{2p}{2-\al}}^p$ is subsumed on the left-hand side. Furthermore, by a Gagliardo-Nirenberg's, with $\si=\f{p(p-1)}{(p+1)(p-2+\al)}$, 
\begin{eqnarray*}
	\|\nabla \theta\|_{L^{p+1}}^{p+1} \leq \|\nabla \theta\|_{L^{\f{2p}{2-\al}}}^{(p+1)\si}  \|\nabla \theta\|_{L^2}^{(p+1)(1-\si)}\leq \f{C_p}{2} \|\nabla \theta\|_{L^{\f{2p}{2-\al}}}^{p}+E_p \|\nabla \theta\|_{L^2}^{\f{(p+1)p(1-\si)}{p-(p+1)\si}}
\end{eqnarray*}
	All in all, taking into account that $p>\si(p+1)$, we obtain 
	$$
	\p_t \|\nabla \theta\|_{L^p}^p\leq D_p	\|\nabla f\|_{L^{\f{2p}{2+\al p-\al}}} ^p+E_p \|\nabla \theta\|_{L^2}^{\f{(p+1)p(1-\si)}{p-(p+1)\si}}\leq D_p	\|\nabla f\|_{L^{\f{2p}{2+\al p-\al}}}^p+ G_p (1+t)^{\f{(p+1)p(1-\si)}{2(p-(p+1)\si)}}
	$$
	Integrating the last inequality in time yields the bound $\|\nabla \theta(t, \cdot)\|_{L^p}\leq C_p(1+t)^{A_{p, \al}}$, with 
$$
A_{p,\al}= \frac{1}{p}+\f{(p+1)(1-\si)}{2(p-(p+1)\si)},
$$
 which is \eqref{75}. Note that for large $p>>1$, we have that $1-\si=O(p^{-1})$, while  $(p-(p+1)\si)=O(1)$. All in all, for $p>>1$, $A_{p, \al}=A_\al+O(p^{-1})$ for some  $A_\al>0$. 
\end{proof}
Note that since $\theta=\tilde{\theta}+v$, we have from \eqref{75} (and under the assumptions of Lemma \ref{le:4})  that 
$$
\|\nabla v(t, \cdot)\|_{L^p}\leq \|\nabla \theta(t, \cdot)\|_{L^p}+\|\nabla \tilde{\theta}\|_{L^p}\leq C(1+t)^{A_{p, \al}}. 
$$
Translating  via \eqref{26}, we obtain, $\|\La  V(\tau, \cdot)\|_{L^p}\sim \|\nabla V(\tau, \cdot)\|_{L^p}\leq C e^{\tau\left(1+A_{p, \al} -\f{1}{\al}-\f{2}{\al p}\right)}$. Using the \\ 
Gagliardo-Nirenberg's inequality and the estimate \eqref{125}, we obtain 
$$
\|\La^s V(\tau, \cdot)\|_{L^p}\leq \|\La  V(\tau, \cdot)\|_{L^p}^s \| V(\tau, \cdot)\|_{L^p}^{1-s} \leq C_p 
e^{\tau\left[(1-\f{3}{\al})+s(\f{2}{\al}+A_{p, \al} -\f{2}{\al p})\right]}. 
$$
The point here is that choosing $s>\f{2}{p}$, say $s=\f{3}{p}$ (so that $\|\La^s V(\tau, \cdot)\|_{L^p}$ controls $\|V\|_{L^\infty}$),  and for sufficiently large $p>p_\al$ (recall $A_{p, \al}=A_\al+O(p^{-1})$),   we can ensure that  the exponent above may be made as close as possible to $e^{(1-\f{3}{\al})\tau}$. We collect these findings  in the following corollary. 
\begin{cor}
	\label{cor:2} 
Let  the assumptions in Lemma \ref{lem7} and Lemma  \ref{le:4}   be satisfied. Then, for arbitrary $\de>0$,  
 there exists $C$ depending on 	$\al, \tilde{\theta}, v_0, \de, f$, so that 
\begin{equation}
\label{112} 
\|V(\tau, \cdot)\|_{L^\infty}+\|\ku_V(\tau, \cdot)\|_{L^\infty}\leq C e^{(1-\f{3}{\al}-\de)\tau}. 
\end{equation}
\end{cor}

\section{A posteriori estimates in $L^2(m)$ spaces} 
\label{sec:6} 
In this section, we establish an asymptotic decomposition  for $V$, which shows that its  main term   
 of $V$ in $L^2(m), 1<m<3-\al$  consists of a simple function of the form $e^{(1-\f{3}{\al})\tau} G$, while the rest of it has faster decay. This is our general plan. However, we follow the scheme outlined in the modified Gronwal's result, Lemma \ref{gro}, which will be applied to  estimate of the remainder term. As one can see from there, we need an {\it a priori} estimate to jump start the process.
  \subsection{A priori estimate in $L^2(2)$} 
 We have already seen in Lemma \ref{lem7} and Corollary \ref{cor:1},  that equation  \eqref{VV1}, has global solutions in $L^p, 1<p<\infty$. Since our arguments in this section necessarily take place  in the  smaller space  $L^2(2)$, we first need to know  well-posedness as well as some \emph{a priori} estimates in this space. In fact,  even if the initial data is well-localized, say  $V(0, \cdot) \in L^2(2)$, it is not {\it a priori} clear why the solution $V(\tau)$ should  stay in $L^2(2)$ for (any) later time $\tau>0$. 
 
 	\begin{prop}
 		\label{thm1}
 	In addition to the standing assumptions about $f$ in Proposition \ref{t:10}, 
 	suppose that $1<m<3-\al$ and $V_0 \in L^\infty\cap L^2(m)(\rtwo)$. Assume that $\tilde{\theta}$ obeys the smallness assumption in Lemma \ref{lem7} and $\tilde{\theta}\in L^2(m)$.   Then \eqref{VV1} has an unique global strong solution 
 	$V \in C^0([0, \infty]; L^2(m))$,  with $V(0)= V_0$. 
 	In addition,  there is the {\it a priori} estimate for each $\de>0$, 
 	\begin{equation}
 	\label{200} 
 	\|V(\tau)\|_{L^2(m)} \leq    C_{\de}
 	e^{\max[m+1-\f{m+4}{\al}-\de, 1-\f{3}{\al}]\tau}. 
 	\end{equation} 
 	where $C$ depends on $\de, V_0, \al, \tilde{\theta}$. 
 \end{prop}
 {\bf Remark:} The estimate in \eqref{200}, while not very inefficient serves only as a preliminary bound, which we feed into the generalized Gronwall's lemma, Lemma \ref{gron}. This eventually  helps us establish the sharp bounds, see Proposition \ref{lem_00} below. 
 \begin{proof}
 We need  control the quantity  $J(\tau):=\int_{\rtwo} (M+|\eta|^{2m})  |V(\tau, \eta)|^2 d\eta$, where $M$ will be selected sufficently large, for technical reasons. For the $L^2$ portion of the quantity, we use the energy inequality \eqref{116} established in Lemma \ref{lem7}, where we note that we can add $\|\La^{\f{\al}{2}} V\|_{L^2}^2$ on the left-hand side. We record it as follows - for any $C>0$, there is a $c_0>0$ and $C_1=C_1(C)$, so that 
 	\begin{equation}
 	\label{1118} 
 	 \partial_{\tau} \|V\|^2_{L^2}+ 2(\f{2}{\al } - 1+C) \|V\|^2_{L^2}+ \|\La^{\f{\al}{2}} V\|_{L^2}^2 
 	 \leq C_1 e^{2 \tau(1-\f{3}{\al})}. 
 	\end{equation}

 To this end, we find the inner product of equation \eqref{VV1} with $|\eta|^{2m} V$. Thus
 	\begin{eqnarray*}
 		& & \f{1}{2} \p_\tau  \int  |\eta|^{2m} V^2  d \eta+ \dpr{|\eta|^{2m} \La^{\al} V}{V} = 
 		(1-\f{1}{\al}) \int |\eta|^{2m} V^2 d \eta+ \f{1}{\al} \int (\eta \cdot \nabla_{\eta} V) |\eta|^{2m} V\ d \eta\\
 		&&- \int (\ku_V \cdot \nabla_{\eta} V) |\eta|^{2m} V\ d \eta-  \int (\ku_{\Theta} \cdot \nabla_{\eta} V) |\eta|^{2m} V\ d \eta-  \int (\ku_V \cdot \nabla_{\eta} \Theta) |\eta|^{2m} V\ d \eta.
 	\end{eqnarray*}
 	We first analyze the terms on the right hand-side. For the term $\int (\ku_V \cdot \nabla_{\eta} \Theta) |\eta|^{2m} V\ d \eta$, we use H\"older's, \eqref{112} and 	 
 	$\||\eta|^{m} \nabla_\eta \Theta\|_{L^2} =e^{(1-\f{m+1}{\al})\tau} \||x|^m \nabla_x \tilde{\theta}\|_{L^2}$,   
 	to conclude that for every $\eps>0$, there is $C_\eps$, 
 	\begin{eqnarray*}
 		|\int (\ku_V \cdot \nabla_{\eta} \Theta) |\eta|^{2m} V\ d \eta| &\leq &  \|\ku_V\|_{L^\infty} \||\eta|^m \nabla_\eta \Theta\|_{L^2} \||\eta|^m V\|_{L^2}\leq C e^{(2-\f{m+4}{\al}-\de)\tau} \sqrt{J(\tau)} \leq \\
 		&\leq & \eps J(\tau)+C_\eps e^{(4-\f{2m+8}{\al}-\de)\tau} 
 	\end{eqnarray*}
 	Next, integration by parts yields 
 	$$
 	\f{1}{\al} \int (\eta \cdot \nabla_{\eta} V) |\eta|^{2m} V d \eta  =   - \f{m+1}{\al} \int |\eta|^{2m} V^2 d \eta= - \f{m+1}{\al} J(\tau).
 	$$
 	For the remaining two terms on the right-hand side of  the energy estimate, we use the divergence free property of $U_V$ and $U_{\Theta}$, as well as integration by parts, and get 
 	\begin{eqnarray*}
 		\int (\ku_V \cdot \nabla_{\eta} V) |\eta|^{2m} V\ d \eta &=&  -m \int |\eta|^{2m-2} (\eta \cdot \ku_V)  V^2 d \eta,\\
 	 	\int (\ku_{\Theta} \cdot \nabla_{\eta} V) |\eta|^{4} V\ d \eta &=&  -m \int |\eta|^{2m-2} (\eta \cdot \ku_{\Theta})  V^2 d \eta.
 	\end{eqnarray*}
 	In the last two  expressions, we need to control quantities in the form 
 	 $\int |\eta|^{2m-1}  |\ku_Q|  V^2(\eta) d \eta$, where $Q$ is either $V$ or $\Theta$. We estimate by H\"older's and Young's inequalities,  for each $\ka>0$, 
 	 \begin{eqnarray*}
 	 \int |\eta|^{2m-1}  |\ku_Q|  V^2(\eta) d \eta &\leq &  C \|\ku_Q\|_{L^{2m}} \left(\int |\eta|^{2m} V^{2} d\eta\right)^{\f{2m-1}{2m}} \|V\|_{L^\infty}^{\f{1}{m}} \leq \\
 	 &\leq &    C  \|Q\|_{L^{2m}} (\ka J(\tau) +C \ka^{-(2m-1)}  \|V\|_{L^{\infty}}^2).
 	 \end{eqnarray*}
 	 Applying this to $Q=V$ and then to $Q=\Theta$ leads to an estimate of the right hand side of the energy estimate as follows 
 	 \begin{equation}
 	 \label{625} 
 	 C (\ka J(\tau) +\ka^{-(2m-1)}  \|V\|_{L^{\infty}}^2) (\|V\|_{L^{2m}}+\|\Theta(\tau)\|_{L^{2m}}).
 	 \end{equation}
 	 On the other hand, by Lemma \ref{lem7}, $\|V\|_{L^{2m}}\leq C e^{(1-\f{3}{\al})\tau}$ and by  \eqref{112}, 
 	 $\|V\|_{L^\infty}\leq C_\de e^{(1-\f{3}{\al}-\de)\tau}$, 
 	  while the estimate for $\Theta$ is much less favorable,   
 	  $\|\Theta(\tau)\|_{L^{2m}}\leq C e^{(1- \f{1}{\al}\left(1+\f{1}{m}\right))\tau}$, according to \eqref{40}. Note that the exponent $e^{(1-\f{1}{\al}\left(1+\f{1}{m}\right))\tau}$ grows, unless $\al<1+\f{1}{m}$. 
 	  
 	  Adding the estimates for $\p_\tau \int |\eta|^{2m} V^2 d\eta$ and the estimate\footnote{(which we 
 	  	multiply by a large constant $M$ and we take $C$ large so that $\f{2}{\al}-1+C>\f{m+2}{\al}-1$} \eqref{1118}  yields 
 	 \begin{eqnarray*}
 	 & &  \f{1}{2} J'(\tau) + \left(\f{m+2}{\al}-1\right) J(\tau)+ M\|\La^{\f{\al}{2}} V\|_{L^2}^2+ 
 	  \dpr{|\eta|^{2m} \La^{\al} V}{V} \leq \\
 	  &\leq & C e^{(1-\f{1}{\al}\left(1+\f{1}{m}\right))\tau}\ka  J(\tau)  + 
 	  C_\de \ka^{-(2m-1)} e^{\left( 3-\f{7+\f{1}{m}}{\al}-\de\right)\tau}+M e^{2\left(1-\f{3}{\al}\right)\tau} 
 	 \end{eqnarray*} 
 	  which is valid for all $\de>0, \ka>0$. 
 	  
 	 Now, we are free to select $\ka$.  We do it so that we can allow ourselves to hide the term containing $J(\tau)$, that is for an arbitrary $\eps$, choose 
 	 $\ka:=\eps e^{-(1-\f{1}{\al}\left(1+\f{1}{m}\right))\tau}$. This brings about the following estimate for $J$, 
 	  \begin{equation}
 	  \label{635}
 	  \f{1}{2} J'(\tau) + \left(\f{m+2}{\al} - 1- \eps\right) J(\tau)+ \dpr{|\eta|^{2m} \La^{\al} V}{V} \leq C_{\de, \eps}  
 	  e^{ \left(2m+2-\f{2m+8}{\al}-\de\right)\tau}+M e^{2\left(1-\f{3}{\al}\right)\tau} .
 	  \end{equation}
 	  It remains to estimate the term  $\dpr{|\eta|^{2m} \La^{\al} V}{V}= 
 	  \dpr{|\eta|^{m} \La^{\al} V}{|\eta|^{m} V}$. Note that this introduces commutators in our estimates as follows 
 	  \begin{eqnarray*}
 	  \dpr{|\eta|^{m} \La^{\al} V}{|\eta|^{m} V}&=&  \dpr{\La^{\f{\al}{2}} |\eta|^{m} \La^{\f{\al}{2}} V}{|\eta|^{m} V} - \dpr{[\La^{\f{\al}{2}},|\eta|^{m}] \La^{\f{\al}{2}} V}{|\eta|^{m} V}.
 	  \end{eqnarray*}
 	  But 
 	  \begin{eqnarray*}
 	   \dpr{\La^{\f{\al}{2}} |\eta|^{m} \La^{\f{\al}{2}} V}{|\eta|^{m} V} &=&  
 	   \dpr{|\eta|^{m} \La^{\f{\al}{2}} V}{\La^{\f{\al}{2}} |\eta|^{m} V}= \dpr{|\eta|^{m} \La^{\f{\al}{2}} V}{|\eta|^{m} \La^{\f{\al}{2}}  V}+\\
 	   &+&  \dpr{|\eta|^{m} \La^{\f{\al}{2}} V}{[\La^{\f{\al}{2}},|\eta|^{m}] V} =  \||\eta|^{m} \La^{\f{\al}{2}} V\|_{L^2}^2 + \dpr{|\eta|^{m} \La^{\f{\al}{2}} V}{[\La^{\f{\al}{2}},|\eta|^{m}] V}.
 \end{eqnarray*}
 Since, by Lemma \ref{L_-50} and Gagliardo-Nirenberg's 
   \begin{eqnarray*}
 	& &   |\dpr{[\La^{\f{\al}{2}},|\eta|^{m}] \La^{\f{\al}{2}} V}{|\eta|^{m} V}| \leq   \|[\La^{\f{\al}{2}},|\eta|^{m}] \La^{\f{\al}{2}} V\|_{L^2} \||\eta|^{m} V\|_{L^2}\\
 	&\leq & C \||\eta|^{m-\f{\al}{2}} \La^{\f{\al}{2}} V\|_{L^2}\|_{L^2} \||\eta|^{m} V\|_{L^2} \\
 	  &\leq & C \sqrt{J(\tau)} \||\eta|^{m} \La^{\f{\al}{2}} V\|_{L^2}^{1-\f{\al}{2m}} \|\La^{\f{\al}{2}} V\|_{L^2}^{\f{\al}{2m}}\leq \eps J(\tau) + \eps \||\eta|^{m} \La^{\f{\al}{2}} V\|_{L^2}^2+C_\eps \|\La^{\f{\al}{2}} V\|_{L^2}^2 \\
 	  & & \dpr{|\eta|^{m} \La^{\f{\al}{2}} V}{[\La^{\f{\al}{2}},|\eta|^{m}] V}\leq \||\eta|^{m} \La^{\f{\al}{2}} V\|_{L^2} \|[\La^{\f{\al}{2}},|\eta|^{m}] V\|_{L^2} \leq   \||\eta|^{m} \La^{\f{\al}{2}} V\|_{L^2} \||\eta|^{m-\f{\al}{2}} V\|_{L^2} \\
 	  &\leq & \eps  \||\eta|^{m} \La^{\f{\al}{2}} V\|_{L^2}^2+ \eps J(\tau)+ C_\eps \|V\|_{L^2}^2.
 	 \end{eqnarray*}
 	  Collecting all the estimates for $\dpr{|\eta|^{2m} \La^{\al} V}{V}$ and using the bound \eqref{115}, yields 
 	  $$
 	  \dpr{|\eta|^{2m} \La^{\al} V}{V}\geq (1-2\eps)   \||\eta|^{m} \La^{\f{\al}{2}} V\|_{L^2}^2-2\eps J(\tau)- C_\eps \|\La^{\f{\al}{2}} V\|_{L^2}^2 - C e^{2\left(1-\f{3}{\al}\right)\tau}. 
 	  $$
 	 This means that for all $\eps>0$, we can derive the energy inequality from \eqref{635}, 
 	  $$
 	  \f{1}{2} J'(\tau) + \left(\f{m+2}{\al} - 1- 3 \eps\right) J(\tau) + (M-C_\eps) \|\La^{\f{\al}{2}} V\|_{L^2}^2 \leq C_{\de, \eps}  
 	  e^{ \left(2m+2-\f{2m+8}{\al}-\de\right)\tau}+M e^{2\left(1-\f{3}{\al}\right)\tau} .
 	 $$
 	 At this point, we make the selection $M=M_\eps=\max(C_\eps,1)$. So, we obtain 
 	 \begin{equation}
 	 \label{637}
 	 \f{1}{2} J'(\tau) + \left(\f{m+2}{\al} - 1- 3 \eps\right) J(\tau)   \leq C_{\de, \eps}  
 	 e^{ \left(2m+2-\f{2m+8}{\al}-\de\right)\tau}+M_\eps e^{2\left(1-\f{3}{\al}\right)\tau} .
 	 \end{equation}
 	  Using integrating  factors, we get the bound 
 	  $$
 	  J(\tau)\leq J(0) e^{2(1-\f{m+2}{\al}+3\eps)\tau}+ C_{\eps, \de}
 	  e^{\max[2m+2-\f{2m+8}{\al}-\de, 2(1-\f{3}{\al})]\tau}. 
 	  $$
 	  Thus, fixing sufficiently small $\eps$, we have that $ (1-\f{m+2}{\al}+3\eps)<1-\f{3}{\al}$, we arrive at the bound 
 	  $$
 	  \left(	\int (1+|\eta|^{2m}) V^2(\tau, \eta) d\eta \right)^{\f{1}{2}} \leq 
 	  C_{\de}
 	  e^{\max[m+1-\f{m+4}{\al}-\de, 1-\f{3}{\al}]\tau}. 
 	  $$
 	  as announced in \eqref{200}. 
 \end{proof}
 
\subsection{Estimate of the remainder}
We first introduce the remainder term. More precisely, we decompose the function $V(\eta, \tau)$ on the spectrum of the operator $\cl$,
\begin{eqnarray}\label{dec}
V= \al(\tau) G+ \widetilde{V},
\end{eqnarray}
where $\al(\tau)= \langle V, 1\rangle$ and $\widetilde{V}= \mathcal{Q}_0 V$. Then,
$$
\al_{\tau}(\tau) =  \langle V_{\tau}, 1\rangle =  \langle \cl V, 1\rangle - \langle U_V \cdot \nabla V, 1\rangle - \langle U_{\Theta} \cdot \nabla V, 1\rangle- \langle U_V \cdot \nabla \Theta, 1\rangle =  (1- \f{3}{\al}) \al(\tau),
$$
since $\cl^*[1]=(1-\f{3}{\al})$. 
This ordinary differential equation for $\al(\tau)$ has the solution $\al(\tau)= \al(0) e^{(1- \f{3}{\al}) \tau}$, where 
\begin{equation}
\label{320} 
\al(0)=\int_{\rtwo} V(\eta) d\eta= \int_{\rtwo} (\theta_0(x)-\tilde{\theta}(x)) dx. 
\end{equation}

We also project the equation \eqref{VV1} on the essential spectrum of the operator $\cl$, i.e 
\begin{equation}
\widetilde{V}_{\tau}= \cl  \widetilde{V} - \mathcal{Q}_0 (\ku_V \cdot \nabla V) - \mathcal{Q}_0 (\ku_{\Theta} \cdot \nabla V)- \mathcal{Q}_0 (\ku_V \cdot \nabla \Theta).
\end{equation} 
Then, $\widetilde{V}$ has the following integral representation
\begin{eqnarray*}
& & \widetilde{V}(\eta, \tau) =  e^{\tau \cl } \widetilde{V}_0 - \int_0^{\tau} e^{(\tau- s) \cl } \mathcal{Q}_0 \nabla 
\bigg[\ku_V \cdot V+ \ku_{\Theta} \cdot V+ \ku_V \cdot \Theta\bigg] ds= \\
&=&  e^{\tau \cl} \widetilde{V}_0 - \int_0^{\tau} e^{(\tau- s) \cl } \mathcal{Q}_0 \nabla \left[(\al(s)\ku_G+\ku_{\tilde{V}})  \cdot (\al(s)G+\tilde{V})\right]ds \\
&-& \int_0^{\tau} e^{(\tau- s) \cl } \mathcal{Q}_0 \nabla \left[\ku_{\Theta} \cdot (\al(s) G+\tilde{V})\right] ds -  \int_0^{\tau} e^{(\tau- s) \cl } \mathcal{Q}_0  \left[ (\al(s) \ku_G+\ku_{\tilde{V}})  \cdot \nabla \Theta\right] ds 
\end{eqnarray*}
where we have used the divergence free property of $\ku_V$ and $\ku_{\Theta}$. We are now ready for the main technical result of this section.  
\begin{prop}
	\label{lem_00}
	Assume  $V_0 \in L^{\infty} \cap L^2(m)$, $1<m<3-\al$.  Then, for any $\eps>0$,   there exists a $C$, depending on $m, \al, \tilde{\theta}, v_0$,so that for $\al_0$ is introduced in \eqref{320} and  for any $\tau> 0$, there is the bound 
	\begin{equation}
	\label{500}
	\|V(\cdot, \tau)- \al_0 e^{(1- \f{3}{\al}) \tau} G(\cdot)\|_{L^2(m)}  \leq C  e^{(2-\f{m+4}{\al}+\eps)\tau}.
	\end{equation}
\end{prop}
Let us comment right away that \eqref{500}, properly interpreted, is nothing but the main claim in Theorem \ref{theo:30}. 
\begin{proof}(Proposition \ref{lem_00}) 
The  main object of investigation  here is  the quantity  $I(\tau):=\|\tilde{V}(\tau)\|_{L^2(m)}$.  We will estimate it in a way that fits the framework of the modified Gronwall's tool, Lemma \ref{gro}. 
We start with the free term, which  is easy to estimate by \eqref{16}, 
$$
\|e^{\tau \cl} \widetilde{V}_0\|_{L^2(m)}\leq C e^{(1- \f{m+2}{\al}+ \eps) \tau}\|f\|_{L^2(m)},
$$
 according to \eqref{18}. Next, by means of \eqref{32} (with $|\ga|=1$),  and H\"older's inequality 
		\begin{eqnarray*}
& &   \int_0^{\tau}\| e^{ (\tau- s) \cl}  \mathcal{Q}_0 \nabla  \left[(\al(s)\ku_G+\ku_{\tilde{V}})  \cdot (\al(s)G+\tilde{V})\right]  \|_{L^2(m)}ds \leq \\
&\leq & C  \int_0^{\tau} \f{e^{(1-\f{m+3}{\al}+\eps)(\tau-s)}}{a(\tau-s)^{\f{1}{\al}}} \|\left[(\al(s)\ku_G+\ku_{\tilde{V}})  \cdot (\al(s)G+\tilde{V})\right]  \|_{L^2(m)}ds \leq \\
&\leq & C \al^2(0)  \|\ku_G\|_{L^\infty} \|G\|_{L^2(m)} \int_0^\tau \f{e^{(1-\f{m+3}{\al}+\eps)(\tau-s)}}{a(\tau-s)^{\f{1}{\al}}} 
e^{2(1-\f{3}{\al}) s} ds + \\
&+& C \al(0) \|\ku_G\|_{L^\infty} \int_0^\tau \f{e^{(1-\f{m+3}{\al}+\eps)(\tau-s)}}{a(\tau-s)^{\f{1}{\al}}}  e^{(1-\f{3}{\al}) s} \|\tilde{V}(s)\|_{L^2(m)} ds + \\
&+& C \al(0) \|G\|_{L^2(m)} \int_0^\tau \f{e^{(1-\f{m+3}{\al}+\eps)(\tau-s)}}{a(\tau-s)^{\f{1}{\al}}}  e^{(1-\f{3}{\al}) s}\|\ku_{\tilde{V}}(s)\|_{L^\infty} ds + \\
&+& C   \int_0^\tau \f{e^{(1-\f{m+3}{\al}+\eps)(\tau-s)}}{a(\tau-s)^{\f{1}{\al}}}  
 \|\ku_{\tilde{V}}(s)\|_{L^\infty} \|\tilde{V}(s)\|_{L^2(m)}  ds. 
		\end{eqnarray*}
	Due to the estimates \eqref{112} and $\al\in (1,2)$, we have that the previous expression is bounded by 
	$$
	C \int_0^\tau \f{e^{(1-\f{m+3}{\al}+\eps)(\tau-s)}}{a(\tau-s)^{\f{1}{\al}}} e^{(1-\f{3}{\al}) s}\left\{ e^{[(1-\f{3}{\al})-\de] s}+ \|\tilde{V}(s)\|_{L^2(m)}  \right\} ds. 
	$$
	The first term is estimated,  due to \eqref{318}, $m<3-\al$ and sufficiently small $\de>0$, 
	$$
	\int_0^\tau \f{e^{(1-\f{m+3}{\al}+\eps)(\tau-s)}}{a(\tau-s)^{\f{1}{\al}}} e^{(1-\f{3}{\al}) s} e^{[(1-\f{3}{\al})-\de] s}  ds\leq 
	C _\eps e^{(1-\f{m+3}{\al}+\eps)\tau}.
	$$
	All in all, 
	\begin{eqnarray*} 
& & 	\|\int_0^{\tau} e^{(\tau- s) \cl } \mathcal{Q}_0 \nabla \left[(\al(s)\ku_G+\ku_{\tilde{V}})  \cdot (\al(s)G+\tilde{V})\right]ds\|_{L^2(m)} \leq   C _\eps e^{(1-\f{m+3}{\al}+\eps)\tau}+ 	\\
	&+& 
	C \int_0^\tau \f{e^{(1-\f{m+3}{\al}+\eps)(\tau-s)}}{a(\tau-s)^{\f{1}{\al}}} e^{(1-\f{3}{\al}) s} \|\tilde{V}(s)\|_{L^2(m)}   ds. 
	\end{eqnarray*}
	Next, we control the other term in the expression for $\tilde{V}$. We have, again by \eqref{32}, 
		  \begin{eqnarray*}
		  	& &   \int_0^{\tau}\| e^{ (\tau- s) \cl}  \mathcal{Q}_0 \nabla   [\ku_{\Theta} \cdot (\al(s) G+\tilde{V}) ]\|_{L^2(m)} ds  \leq \\
		  	&\leq &   C   \int_0^\tau \f{e^{(1-\f{m+3}{\al}+\eps)(\tau-s)}}{a(\tau-s)^{\f{1}{\al}}} |\al(0)| e^{(1-\f{3}{\al}) s}
		  	\|\ku_\Theta G\|_{L^2(m)}  ds +   
		  		C  \int_0^\tau \f{e^{(1-\f{m+3}{\al}+\eps)(\tau-s)}}{a(\tau-s)^{\f{1}{\al}}} 
		  	\|\ku_\Theta\tilde{V}(s)\|_{L^2(m)}  ds. 
		  \end{eqnarray*}
Note that 
	$$
		 \|\ku_\Theta(s) G\|_{L^2(m)}\leq C \|\ku_\Theta(s)\|_{L^2} \|(1+|\cdot|^m)G\|_{L^\infty} \leq C \|\Theta\|_{L^2} \leq C e^{(1-\f{2}{\al})s}, 
$$ 
	while by the Sobolev embedding \eqref{a:40}
		\begin{eqnarray*}
	 	\|\ku_\Theta\tilde{V}(s)\|_{L^2(m)} &\leq & C 	\|\tilde{V}(s)\|_{L^2(m)} \|\ku_\Theta(s)\|_{L^\infty} \leq  C_\de\|\tilde{V}(s)\|_{L^2(m)}
	 	(\|\La^{-\de} \nabla \ku_\Theta\|_{L^2}+ \|\La^{\de} \nabla \ku_\Theta\|_{L^2})\\
	 	&\leq & C_\de\|\tilde{V}(s)\|_{L^2(m)} 	(\|\La^{-\de} \Theta\|_{L^2}+ 
	 	\|\La^{\de} \Theta\|_{L^2})\leq  C_\de e^{(1-\f{2-\de}{\al})s} \|\tilde{V}(s)\|_{L^2(m)}.
		\end{eqnarray*}
	All in all, choosing $\de<2-\al$, say $\de=\f{2-\al}{2}$, applying \eqref{318} and $1<m<3-\al$ and $\eps<<1$,  we obtain the bound 
		\begin{eqnarray*}
& & 	\| \int_0^{\tau} e^{ (\tau- s) \cl}  \mathcal{Q}_0 \nabla   [\ku_{\Theta} \cdot (\al(s) G+\tilde{V}) ]ds \|_{L^2(m)} \leq   C \int \f{e^{(1-\f{m+3}{\al}+\eps)(\tau-s)}}{a(\tau-s)^{\f{1}{\al}}}  e^{(2-\f{5}{\al})s} ds +	\\
	&+& C \int \f{e^{(1-\f{m+3}{\al}+\eps)(\tau-s)}}{a(\tau-s)^{\f{1}{\al}}}   e^{(\f{1}{2}-\f{1}{\al})s} \|\tilde{V}(s)\|_{L^2(m)} ds\leq \\
	&\leq & C e^{(2-\f{5}{\al})\tau}+C \int \f{e^{(1-\f{m+3}{\al}+\eps)(\tau-s)}}{a(\tau-s)^{\f{1}{\al}}}   e^{(\f{1}{2}-\f{1}{\al})s} \|\tilde{V}(s)\|_{L^2(m)} ds. 
 	\end{eqnarray*}
	Next, we estimate the contribution of the last two terms in the equation for $\tilde{V}$. We have 
	\begin{eqnarray*}
& & 	\|	\int_0^{\tau} e^{ (\tau- s) \cl}  \mathcal{Q}_0 \nabla 
		(\al(s) \ku_G\cdot \Theta(s)) ds\|_{L^2(m)} \leq 	C |\al(0)| \int_0^{\tau} e^{ (1-\f{m+3}{\al}+\eps)(\tau- s)} e^{(1-\f{3}{\al})s} \|\ku_G\Theta(s)\|_{L^2(m)}  \\
		&\leq & C |\al(0)|\|\ku_G\|_{L^\infty}  \int_0^{\tau} e^{ (1-\f{m+3}{\al}+\eps)(\tau- s)} e^{(1-\f{3}{\al})s}  \|\Theta(s)\|_{L^2(m)} ds\leq C  \int_0^{\tau} e^{ (1-\f{m+3}{\al}+\eps)(\tau- s)} e^{(2-\f{5}{\al})s} ds\\
		&\leq & C e^{ (1-\f{m+3}{\al}+\eps)\tau}.
	\end{eqnarray*}
 Finally, we estimate the contribution of 
 $
 \int_0^{\tau} e^{ (\tau- s) \cl}  \mathcal{Q}_0 [\ku_{\tilde{V}}  \cdot \nabla  \Theta] ds, 
 $ it turns out that we need to split it as follows 
	\begin{eqnarray*}
\int_0^{\tau} e^{ (\tau- s) \cl}  \mathcal{Q}_0 [\ku_{\tilde{V}}  \cdot \nabla  \Theta] 
=\int_0^{\tau} e^{ (\tau- s) \cl}  \mathcal{Q}_0 [\ku_{\tilde{V}}  \cdot \chi(\eta) \nabla  \Theta ] 
+\int_0^{\tau} e^{ (\tau- s) \cl}  \mathcal{Q}_0 [\ku_{\tilde{V}}  \cdot (1-\chi(\eta)) \nabla  \Theta ],  
	\end{eqnarray*}
	where $\chi\in C^\infty_0$ is supported in $|\eta|<1$. In the region $|\eta|<1$, we have the bound 
		\begin{eqnarray*}
& & 		\|	\int_0^{\tau} e^{ (\tau- s) \cl}  \mathcal{Q}_0 [\ku_{\tilde{V}}  \cdot \chi(\eta) \nabla  \Theta] ds\|_{L^2(m)} \leq 		\int_0^{\tau} \|e^{ (\tau- s) \cl}  \mathcal{Q}_0 \nabla  [\ku_{\tilde{V}}  \cdot \chi(\eta)   \Theta] \|_{L^2(m)} ds + \\
		&+& 	\int_0^{\tau} \|e^{ (\tau- s) \cl}  \mathcal{Q}_0   [\ku_{\tilde{V}}  \cdot \Theta \nabla \chi(\eta)   ] \|_{L^2(m)} ds. 
		\end{eqnarray*}
	We apply either \eqref{16} or \eqref{32} to obtain 
		\begin{eqnarray*}
			& & 		\|	\int_0^{\tau} e^{ (\tau- s) \cl}  \mathcal{Q}_0 [\ku_{\tilde{V}}  \cdot \chi(\eta) \nabla  \Theta] ds\|_{L^2(m)} \leq 	
			C \int_0^{\tau} e^{ (1-\f{m+3}{\al}+\eps)(\tau- s)} \|\ku_{\tilde{V}}\|_{L^\infty}
			 \|\chi(\eta)\Theta\|_{L^2(m)} ds + \\
			 &+& C \int_0^{\tau} 
			 e^{ (1-\f{m+2}{\al}+\eps)(\tau- s)}  \|\ku_{\tilde{V}}\|_{L^\infty}
			 \|\nabla \chi(\eta) \Theta\|_{L^2(m)} ds\leq  \\
			 &\leq &  C \int_0^{\tau} 
			 e^{ (1-\f{m+2}{\al}+\eps)(\tau- s)}  e^{(1-\f{3}{\al}-\de)s} 
			 \|\Theta(s)\|_{L^2} ds \leq \\
			 &\leq &  C \int_0^{\tau} 
			 e^{ (1-\f{m+2}{\al}+\eps)(\tau- s)}  e^{(2-\f{5}{\al}-\de)s} ds\leq C e^{(1-\f{m+2}{\al}+\eps)\tau}.
		\end{eqnarray*}
	where we have used \eqref{40} and \eqref{112}. 
	
	Finally, in the region $|\eta|>1$, we apply \eqref{32}. We obtain   
		\begin{eqnarray*}
		& &  \int_0^{\tau} \|e^{ (\tau- s) \cl}  \mathcal{Q}_0 [\ku_{\tilde{V}}  \cdot (1-\chi(\eta)) \nabla  \Theta ]\|_{L^2(m)}  ds\leq \\
		&\leq &   \int_0^{\tau} 
		 e^{ (1-\f{m+2}{\al}+\eps)(\tau- s)} \|\ku_{\tilde{V}(s)}\|_{L^\infty} 
		 \|(1-\chi(\eta)) \nabla \Theta\|_{L^2(m)} ds \\
		 &\leq & C  \int_0^{\tau} 
		 e^{ (1-\f{m+2}{\al}+\eps)(\tau- s)} e^{(1-\f{3}{\al}+\de)s} 
		 \||\eta|^m \nabla \Theta\|_{L^2} ds \\
		 &\leq & C  \int_0^{\tau} 
		 e^{ (1-\f{m+2}{\al}+\eps)(\tau- s)} e^{(2-\f{m+4}{\al}+\de)s} ds\leq  C e^{ (2-\f{m+4}{\al}+\de)\tau}.
		\end{eqnarray*}
	where we have used $\||\eta|^m \nabla \Theta(s)\|_{L^2}=e^{(1-\f{m+1}{\al})s} \||\eta|^m \nabla \tilde{\theta}\|_{L^2}$ and $\al<2$. 
	
	Putting all the estimates together implies the {\it a posteriori} bound 
	$$
	\|\tilde{V}(\tau)\|_{L^2(m)}\leq C e^{(2-\f{m+4}{\al}+\de)\tau}+C \int_0^\tau 	 e^{ (1-\f{m+3}{\al}+\eps)(\tau- s)} e^{(\f{1}{2}-\f{1}{\al}) s} \|\tilde{V}(s)\|_{L^2(m)} ds. 
	$$
	Applying the Gronwall's inequality, Lemma \ref{gro}, we obtain the bound 
	$$
	\|\tilde{V}(\tau)\|_{L^2(m)}\leq C e^{(2-\f{m+4}{\al}+\de)\tau}.
	$$
	
\end{proof} 
 \appendix
 
 \section{Proof of Lemma \ref{L_-50}}
 	The proof of this lemma is based by some modifications in the proof of   relation $(4.8)$, \cite{HS}.   Recall, that for $s\in (0,2)$
 	\begin{eqnarray*}
 		[|\nabla|^{s}, g] f (x)&=& |\nabla|^{s} (g f)- g \ |\nabla|^{s}f= c_s \int \f{f(x) g(x)- f(y) g(y)}{|x- y|^{2+ s}} dy- g(x) c_s\int \f{f(x)- f(y)}{|x- y|^{2+ s}} dy \\
 		&=& c_s \int \f{ f(y)( g(x)-  g(y) )}{|x- y|^{2+ s}} dy.
 	\end{eqnarray*}
 	Introduce a smooth partition of unity, that is a function $\psi\in C^\infty_0(\rone)$, $supp\  \psi\subset (\f{1}{2},2)$, so that 
 	$$
 	\sum_{k=-\infty}^\infty \psi(2^{-k} |\eta|) = 1, \eta\in\rtwo, \eta\neq 0.
 	$$
 	Introduce another $C^\infty_0$ function $\Psi(z)=|z|^\si \psi(z)$, so that we can decompose 
 	$$
 	|\eta|^{\si} = \sum_{k=-\infty}^\infty |\eta|^{\si} \psi(2^{-k} |\eta|) = \sum_{k=-\infty}^\infty 2^{k \si} \Psi(2^{-k} |\eta|). 
 	$$
 	We can then write 
 	\begin{eqnarray*}
 		F(\eta) &:=&  [\La^{s},|\eta|^\si] f=   \sum_k 2^{\si k}  [\La^s, \Psi (2^{- k} \cdot)]  f(\eta)  = 
 		\sum_k 2^{\si k}   \int  \f{f(y) (\Psi(2^{-k}\eta)- \Psi(2^{-k}y))}{|\eta- y|^{2+ s}} dy. 
 	\end{eqnarray*}
 	Introducing  
 	$$
 	F_k:= \int  \f{|f(y)| |\Psi(2^{-k}\eta)- \Psi(2^{-k}y)|}{|\eta- y|^{2+ s}} dy, 
 	$$
 	we need to control 
 	\begin{eqnarray*}
 		\|F\|_{L^2}^2 &=&  \sum_l \int_{|\eta|\sim 2^l} |F(\eta)|^2 d\eta=\sum_l \int_{|\eta|\sim 2^l} \left|\sum_k 2^{s k} F_k(\eta)\right|^2 d\eta=\\
 		&=& \sum_l \int_{|\eta|\sim 2^l} \left|\sum_{k>l+10}  2^{s k} F_k(\eta)\right|^2 d\eta+\sum_l \int_{|\eta|\sim 2^l} \left|\sum_{k=l-10}^{l+10}  2^{s k} F_k(\eta)\right|^2 d\eta+\\
 		&+& \sum_l \int_{|\eta|\sim 2^l} \left|\sum_{k<l-10}  2^{s k} F_k(\eta)\right|^2 d\eta=:K_1+K_2+K_3
 	\end{eqnarray*}
 	
 	We first consider the cases $k>l+10$.  One can estimate easily $F_k$ point-wise.  More specifically, since in the denominator of the expression for $F_k$, we have $|\eta-y|\geq \f{1}{2} |\eta|\geq 2^{k-3}$, 
 	$$
 	|F_k(\eta)|\leq 2^{-k(2+\si)} \int  |f(y)||\Psi(2^{-k}y)| dy\leq C 2^{-k(1+\si)} \|f\|_{L^2(|y| \sim 2^k)},
 	$$
 	whence 
 	\begin{eqnarray*}
 		& & K_1 \leq 
 		\sum_l 2^{2 l} \sum_{k_1>l+10} \sum_{k_2>l+10} 
 		2^{k_1(s- 1-\si)} \|f\|_{L^2(|y| \sim 2^{k_1})} 2^{k_2(s- 1- \si)} \|f\|_{L^2(|y| \sim 2^{k_2})}\\
 		&\leq &   \sum_{k_1} \sum_{k_2} 2^{2\min(k_1,k_2)}2^{k_1(s- 1- \si)}  \|f\|_{L^2(|y| \sim 2^{k_1})} 2^{k_2(s- 1- \si)} \|f\|_{L^2(|y| \sim 2^{k_2})} \\
 		&\leq & C \sum_k 2^{2 k(s-\si)}  \|f\|_{L^2(|y| \sim 2^{k})}^2 \leq C  \| |\eta|^{s- \si} f\|^2.
 	\end{eqnarray*}
 	where we have used $\sum_{l: l<\min(k_1,k_2)-10}  2^{2l} \leq C 2^{2\min(k_1,k_2)}$. 
 	
 	For the case $k<l-10$, we perform similar argument, since 
 	$$
 	|F_k(\eta)|\leq C 2^{-l(2+\si)} 2^k \|f\|_{L^2(|y| \sim 2^k)}. 
 	$$
 	So, 
 	\begin{eqnarray*}
 		& & K_3 \leq C \sum_l 2^{2l} 2^{-2 l (2+ \si)} \sum_{k_1<l-10} \sum_{k_2<l-10} 2^{(s+ 1) k_1} \|f\|_{L^2(|y| \sim 2^{k_1})} 2^{(s+ 1) k_2} \|f\|_{L^2(|y| \sim 2^{k_2})} \\
 		&\leq &C  \sum_{k_1} \sum_{k_2} 2^{(s+ 1) k_1} \|f\|_{L^2(|y| \sim 2^{k_1})} 2^{(s+ 1) k_2} \|f\|_{L^2(|y| \sim 2^{k_2})} 2^{-2 (1+ \si)\max(k_1,k_2)} \\
 		&\leq & C \sum_k 2^{2 k(s- \si)}  \|f\|_{L^2(|y| \sim 2^{k})}^2 \leq C  \| |\eta|^{(s- \si)} f\|^2.
 	\end{eqnarray*}
 	Finally, for the case $|l-k|\leq 10$, we use 
 	$$
 	|\Psi (2^{- k} \eta)- \Psi (2^{- k} y)| \leq 2^{- k} |\eta- y| |\nabla \Psi (2^{- k} (\eta- y))| \leq C 2^{- k} |\eta- y|,
 	$$
 	so that 
 	$$
 	|F_k(\eta)|\leq C 2^{-k} \int_{|y| \sim 2^k} \f{|f(y)|}{|\eta- y|^{1+ \si}} dy = C 2^{-k} |f| \chi_{|y| \sim 2^k} * \f{1}{|\cdot|^{1+ \si}}.
 	$$
 	Thus, by H\"older's 
 	\begin{eqnarray*}
 		& & K_2  \leq C \sum_k \int_{|\eta|\sim 2^k} 2^{s k} \left||f| \chi_{|y| \sim 2^k} * \f{1}{|\cdot|^{1+ \si}}\right|^2 d\eta\leq 
 		C \sum_k 2^{s k} \||f| \chi_{|y| \sim 2^k} * \f{1}{|\cdot|^{1+ \si}}\|_{L^2(|\eta|\sim 2^k)}^2 \\
 		&\leq & C \sum_k 2^{2 k(s-\si)} \||f| \chi_{|y| \sim 2^k} * \f{1}{|\cdot|^{1+ \si}}\|_{L^{\f{2}{\si}}(|\eta|\sim 2^k)}^2 \leq 
 		C \sum_k 2^{2 k(s-\si)}   \|f\|_{L^2(|\eta| \sim 2^k)}^2 \leq C  \| |\eta|^{s-\si} f\|^2.
 	\end{eqnarray*}
 	where we have used the Hausdorf-Young's inequality 
 	$$
 	\|f \chi_{|y| \sim 2^k} * \f{1}{|\cdot|^{1+ \f{\al}{2}}}\|_{L^{\f{2}{\si}}}   \leq C \|\f{1}{|\cdot|^{1+ \si}} \|_{L^{\f{2}{1+ \si}, \infty}} \ \|f\|_{L^2(|\eta| \sim 2^k)}\leq C \|f\|_{L^2(|\eta| \sim 2^k)}.
 	$$
{\bf Conflict of interest statement: } On behalf of all authors, the corresponding author states that there is no conflict of interest.

\end{document}